\documentclass[11pt]{amsart}
\usepackage{amsfonts,amssymb,amsmath,amsthm,here}

\theoremstyle{plain}
\newtheorem{theo}{Theorem}[section]

\newtheorem{cor}[theo]{Corollary}
\newtheorem{lem}[theo]{Lemma}
\newtheorem{prop}[theo]{Proposition}

\theoremstyle{definition}

\newtheorem{ex}[theo]{Example}

\newtheorem{rem}[theo]{Remark}

\makeatletter

\@addtoreset{equation}{section}
\makeatother

\makeatletter
\@namedef{subjclassname@2020}{
  \textup{2020} Mathematics Subject Classification}
\makeatother

\def \Z{{\mathbb Z}}
\def \Q{{\mathbb Q}}

\def \Gal{\mathrm{Gal}}

\author{Yu Hashimoto}
\address{Department of Mathematics,
Interdisciplinary Faculty of Science and Engineering,
Shimane University,
Matsue, Shimane, 690-8504, Japan}
\email{hashimoto-yu1@ict.shimanet.ed.jp}

\author{Miho Aoki}
\address{Department of Mathematics,
Interdisciplinary Faculty of Science and Engineering,
Shimane University,
Matsue, Shimane, 690-8504, Japan}
\email{aoki@riko.shimane-u.ac.jp}

\thanks{This work was supported by JSPS KAKENHI Grant Number JP21K03181}

\title[Cyclic quintic fields]{Normal integral bases of \\ Lehmer's  cyclic quintic fields}

\begin{document}
\maketitle

\begin{abstract}
Let $K_n$ be a  tamely ramified cyclic quintic field   generated by a root of Emma Lehmer's parametric polynomial.
We give all normal integral bases  for $K_n$ only by the roots of the polynomial,  which is a generalization of the work of Lehmer in the case 
that  $n^4+5n^3+15n^2+25n+25$ is prime number, 
and Spearman-Willliams 
 in the case that  $n^4+5n^3+15n^2+25n+25$ is square-free.

\end{abstract}

\section{Introduction}\label{sect:intro}

Let $n$ be an integer.
We consider Emma Lehmer's quintic polynomial $f_n(X)\in\Z[X]$ defined by
\begin{align}\label{eq:f_n}
f_n(X)  := & X^5+n^2X^4-(2n^3+6n^2+10n+10)X^3  \notag \\
& +(n^4+5n^3+11n^2+15n+5)X^2+
(n^3+4n^2+10n+10)X+1, 
\end{align}
(see \cite{Leh}).
The polynomial $f_n(X)$ is irreducible for any $n\in\Z$,
and $K_n:=\Q(\rho_n)$ is a cyclic quintic field where $\rho_n$ is a root of 
$f_n(X)$ (see \cite{SW}).
Put $\delta_n := n^3+5n^2+10n+7$ and $\Delta_n:=n^4+5n^3+15n^2+25n+25$.
We have
$(n^3+5n^2+10n+18)\delta_n -(n^2+5n+5)\Delta_n=1$ and hence $(\delta_n, \Delta_n)=1$.
For  a primitive fifth root of unity $\zeta=\zeta_5$, we have $\Delta_n =N_{\Q (\zeta)/\Q} (n+2+2\zeta^4+\zeta^2) ( >0)$,
where $N_{\mathbb Q(\zeta)/\mathbb Q}$ is the norm map from $\mathbb Q (\zeta)$ to $\mathbb Q$.
Let $v_p(x)$ be the $p$-adic valuation of $x \in \Q$ for a prime number $p$.
The discriminant of  $f_n(X)$ is $d(f_n)=\delta_n^2 \Delta_n^4$, 
and the conductor $\mathfrak{f}_{K_n}$ 
of $K_n$ has been determined by Jeannin \cite[Th\'eor\`eme~1]{Je} as follows.

\begin{equation}\label{eq:cond}
\mathfrak{f}_{K_n}=5^c \!\!\!\!\!\!\!\!\!  \prod_{\substack{
p\mid \Delta_n,\, p\ne 5
\\
v_p(\Delta_n )\not\equiv 0 \!\!\!\! \pmod{5}
}} \!\!\!\!\!\!\!\!\! p,\quad 
c=
\begin{cases} 
0 & (\text{if}\ 5\nmid n),
\\ 
2 &  (\text{if}\ 5\mid n),
\end{cases}
\end{equation}
and he proved that  any prime number $p$ dividing $\Delta_n$ satisfies $p=5$ or $p\equiv 1 \pmod{5}$
(\cite[Lemme~2.1.1]{Je}).
Especially, $K_n/\mathbb{Q}$ is tamely ramified if and only if $5\nmid n$.
Note that  it is known from the conductor-disctiminant formula that 
$D_L=\mathfrak{f}_L^{p-1}$ for a cyclic number field $L$ of odd
prime degree $p$ (\cite[Theorem~3.11]{Was}), and hence we have
\begin{equation}\label{eq:disc}
D_{K_n}=\mathfrak{f}_{K_n}^4
\end{equation}

Let $K/\Q$ be a finite Galois extension with  Galois group G.
Let $\mathcal O_K$  denote the ring of integers of $K$.
If the set $\{\sigma(\alpha)\mid\sigma\in G\}$ is a basis of $\mathcal O_K$ 
over $\Z$, we call the basis {\it a normal integral basis}
(NIB), and $\alpha$ the generator.
If $K/\mathbb{Q}$ has an NIB, then the extension is 
tamely ramified. 
For abelian number fields, we have the following (\cite[Chap.~9, Theorem 3.4]{L}).

%\begin{thm}[Hilbert-Speiser]
\begin{theo}[Hilbert-Speiser]
Let $K$ be an abelian number field.
The following three conditions are equivalent. 
\begin{itemize}
\item[(i)]
$K/\Q$ is tamely ramified.
%\\
\item[(ii)]
The conductor of $K$ is square-free. 
%\\
\item[(iii)]
$K/\Q$ has a normal integral basis.
\end{itemize}
%\end{thm} 
\end{theo}
Let $K$ be an abelian number field with  conductor $\mathfrak f$.
We call the conjugates of $\mathrm{Tr}_{\mathbb Q(\zeta_{\mathfrak f})/K} (\zeta_{\mathfrak f})$
{\it Gaussian periods}, where  $\mathrm{Tr}_{\mathbb Q(\zeta_{\mathfrak f})/K}$ is the trace map from 
$\mathbb Q(\zeta_{\mathfrak f})$ to $K$.
A Gaussian period is a generator of an NIB of $K$ if $K/\mathbb Q$ is tamely ramified
(\cite[Proposition 4.31]{N}).
We describe the NIBs of $K_n$ using only the roots of the polynomial $f_n(X)$,
and do not use a root of unity $\zeta_{{\mathfrak f}_{K_n}}$
which depends on the conductor of $K_n$.

From the Hilbert-Speiser  theorem, we know that  $K_n$ has an NIB if and only if 
$5 \nmid n$.
If $\Delta_n=p$ is a prime number, then $\mathfrak f_{K_n}=p$ and $K_n/\Q$ is tamely ramified.
Lehmer \cite[p.~539]{Leh} showed that $\left( \frac{n}{5} \right) \left(\rho_n +\left( n^2-\left( \frac{n}{5} \right)\right)/5 \right)$
is a Gaussian period in $K_n$. Therefore, it follows that $ \rho_n +\left( n^2-\left( \frac{n}{5} \right)\right)/5 $
is a generator of an NIB.
Spearman and Williams \cite{SpWi} proved the following.
\begin{theo}[Spearman-Williams] \label{theo:SW}
Assume that $5\nmid n$.
Then $K_n$ has a generator of a normal integral basis of the form $v+w\rho_n \ (v,w\in \Z)$
if and only if $\Delta_n$ is square-free. Furthermore, in this case, the integers $v$ and $w$ are given
by $w=\pm1$ and $ v=w\left( n^2-\left(\frac{n}{5} \right)\right)/5$, where $\left(\frac{n}{5} \right)$
is the Legendre symbol.
\end{theo}

We put $\Delta_n= ab^2c^3d^4e^5$ with  $a,b,c,d,e \in \Z_{>0}$, where  $a,b,c,d$ are square-free and pairwise coprime,
and Gal$(K_n/\Q) =\langle \sigma \rangle$ where the generator $\sigma$ satisfies $
\sigma (\rho_n) =(n+2+n\rho_n -\rho_n^2)/(1+(n+2)\rho_n)$
(\cite[(3.2)]{SW}). For $i\in \Z/5\Z$, we put
$\rho_n^{(i)} :=\sigma^i (\rho_n)$.  
In this paper, we prove that for some integers $\beta_0, \beta_1, \beta_2, \beta_3$ which satisfy $b^2c^4d^6e^6 =N_{\Q (\zeta_5)/\Q}
(\beta_0 +\beta_1 \zeta_5 +\beta_2 \zeta_5^2+\beta_3 \zeta_5^3)$,
\begin{align*}
& \frac{1}{bc^2d^3e^4} (\beta_0 \rho_n+\beta_1 \rho_n^{(1)} +\beta_2 \rho_n^{(2)} +\beta_3 \rho_n^{(3)} -m), \\
\end{align*}
is a generator of an NIB of $K_n$, where 
\[
 m:=\frac{1}{5} \left( \left( \frac{n}{5} \right) bc^2d^3e^4-n^2(\beta_0+\beta_1+\beta_2+\beta_3) \right)
\] and the integers $\beta_0, \beta_1, \beta_2, \beta_3$ are
given explicitly (see Theorem~\ref{theo:main}).
For example, when $\Delta_n$ is square-free, then  these integers are given by $\beta_0=1,\, \beta_1=\beta_2=\beta_3=0$ (Corollary~\ref{cor:NIB-squarefree}),
and using $b=c=d=e=1$,  we obtain Spearman and Williams' result (Theorem~\ref{theo:SW}).
The smallest positive integer $n$ such that $\Delta_n$ is not  square-free is $n=14$, and  $\Delta_{14}=11\cdot 71^2$. In this case, 
we have $a=11, b=71, c=d=e=1$. The integers are given by $\beta_0=6,\, \beta_1=7,\, \beta_2=8,\, \beta_3=10$ and  $m=-1201$, it follows that
$(6\rho_n +7\rho_n^{(1)}+8\rho_n^{(2)} +10\rho_n^{(3)}+1201)/71$ is  a generator of an NIB of $K_n$ for $n=14$ (Example~\ref{ex:1to1000}).
For the proof of the main theorem (Theorem~\ref{theo:main}),  we  use the method of constructing an NIB from a normal basis and an integral basis,
proposed by  Acciaro and Fieker  in \cite{AF}. 
Therefore,  we first obtain an integral basis for tamely ramified extensions $K_n/\Q$ (Theorem~\ref{theo:IB}).
When $\Delta_n$ has no square factor of a prime other than 5,
Ga\'al and Pohst \cite[Lemma~2]{GP} proved 
that $K_n$ has an integral basis of the form $\{  1, \rho, \rho^2, \rho^3, * \}$. 
When $\Delta_n$ is cube-free, 
Eloff, Spearman and Williams \cite[Theorem]{ESW1} proved that 
$K_n$ has an integral basis of the form $\{  1, \rho, \rho^2, *, * \}$. 
We will give an integral basis for all $n$ with $5\nmid n$.

Furthermore, by acting the elements of $\Z [\mathrm{Gal}(K_n/\Q)]^{\times}$ on the NIB obtained by Theorem~\ref{theo:main},
 we give all the NIBs of $K_n$ (Theorem~\ref{theo:all-NIB}). This is a generalization of the case $ n=-1$ by 
 Davis, Eloff, Spearman and Williams \cite{DES,ESW2}.
Let $K/\mathbb{Q}$ be a finite Galois extension with the Galois group $G$,
and $\mathcal{O}_K$ be the ring of integers of $K$.
The action of the group ring  $\Q[G]$  on $K$ is given by
\[
x\cdot a:=\sum_{\sigma\in G}n_{\sigma}\sigma(a)\ \ 
(x=\sum_{\sigma\in G}n_\sigma\sigma\in\ \Q[G],\ a\in K),
\]
and $K$ (resp. $\mathcal{O}_K$) is a $\Q[G]$
(resp. $\Z[G]$)-module with this action.
There is a one-to-one correspondence between the generators of an NIB of a finite Galois extension $K/\Q$ which has an NIB 
and elements of the
multiplicative group $\mathbb Z[G]^{\times}$.
In fact, if $\alpha$ is a generator of an NIB of $K$, then all  generators of an NIB are  given by $u\cdot \alpha$
for some $u\in \Z[G]^{\times}$
 (\cite[Lemma~2.3]{HA}).
 Let $K_n$ be the cyclic quintic field and 
 $G=\Gal(K_n/\Q)=\langle \sigma \rangle$,  we have  $1-\sigma^2-\sigma^3 \in \mathbb Z[G]^{\times} ,\ (1-\sigma^2-\sigma^3)^{-1}=1-\sigma-\sigma^4,\ \langle 1-\sigma^2-\sigma^3 \rangle \simeq \Z,\, $ and
\begin{equation}\label{eq:ZG}
\Z [G]^{\times} =\langle \pm 1 \rangle \times G \times \langle 1-\sigma^2-\sigma^3 \rangle = \{\pm \sigma^{\ell} (1-\sigma^2-\sigma^3)^k 
\ |\ \ell \in \Z/5  \Z, k\in \Z \},
\end{equation}
 (\cite[p.~2934]{AB}).
 
\begin{rem}\label{re:Shanks}
In the previous paper \cite{HA}, we obtained all NIBs of the cubic cyclic extension $\Q (\rho_n)/\Q$ 
for a root $\rho_n$ of Shanks' cubic polynomial $X^3-nX^2-(n+3)X-1 \, (n\in \Z)$ (namely, the simplest cubic field).
Put $n^2+3n+9= de^2c^3$ with  $d,e,c \in \Z_{>0}$ and $d,e $ being square-free and $(d,e)=1$. 
Let $\sigma$ be the generator of Gal$(\Q (\rho_n)/\Q)$ satisfying $\sigma(\rho_n)=-1/(1+\rho_n)$.
Let  $\rho_n':=\sigma (\rho_n)$ and $\zeta_3$ be a primitive cubic root of unity. 
Then, we proved that for some integers $a_0$ and $a_1$ that satisfy $ec=N_{\Q (\zeta_3)/\Q} (a_0+a_1 \zeta_3) =a_0^2-a_0a_1+a_1^2$,
\[
\alpha :=\frac{1}{ec^2} (a_0 \rho_n+a_1\rho_n'+m)
\]
is a generator of an NIB of $\Q (\rho_n)$, where 
\[
m:=\frac{ (\varepsilon ec^2 -n (a_0+a_1))}{3}, \quad \varepsilon (=\pm 1) := \begin{cases}
\left( \frac{n(a_0+a_1)}{3} \right), & \text{if $3\nmid n$}, \\
\left( \frac{a_0}{3} \right), & \text{if $n\equiv 12 \pmod{27}$}.
\end{cases}
\]
Furthermore, all generators of NIBs are given by $\{ \pm \sigma^{\ell}\cdot \alpha \ |\ \ell \in \Z/3\Z \}$.
We can choose  integers $a_0,a_1$ such that  $ec=a_0^2-a_0a_1+a_1^2$, and $a_0+a_1 \zeta_3 $ divides $ n+3(1+\zeta_3)$
in $\Z [\zeta_3]$, and find them  as follows.  First, we can write $n^2+3n+9=(n+3(1+\zeta_3))(n+3(1+\zeta_3^2))$ and $ec=3^j p_1
\cdots p_k$ with $j=\begin{cases}
0 , & \text{ if } 3\nmid n,\\
1, & \text{  if } n\equiv 12 \pmod{27}
\end{cases}$ and $p_1,\ldots,p_k$ are prime numbers (not necessarily different) with $p_1 \equiv \cdots \equiv p_k \equiv 1\pmod{3}$.
Let $\pi_i$ and $\pi_i'$ be prime elements of $\mathbb Z [\zeta_3 ]$ with $p_i=\pi_i \pi_i'$ and
$\pi_i \, |\, (n+3(1+\zeta_3)),\ \pi_i' \, |\, (n+3(1+\zeta_3^2))$ in $\mathbb Z [\zeta_3]$.
Then the integers $a_0$ and $a_1$ are given by $(1-\zeta_3)^j \pi_1 \cdots \pi_k=a_0+a_1 \zeta_3$
(note that only the  prime ideal containing both $n+3(1+\zeta_3)$ and $n+3(1+\zeta_3^2)$ is the prime ideal $(1-\zeta_3)$ above $3$).
\end{rem}
\section{Integral bases}\label{sec:IB}

In this section, we consider an  integral basis of $K_n$. Let $\rho :=\rho_n$ be a root of 
the quintic polynomial $f_n(X)$ of (\ref{eq:f_n}). Ga\'al and Pohst \cite[Lemma~2]{GP} have shown 
that if $p^2$ does not divide $\Delta_n$ for any prime number $p$ different from $5$, then 
$K_n$ has an integral basis of the form $\{  1, \rho, \rho^2, \rho^3, * \}$. 
Eloff, Spearman and Williams \cite[Theorem]{ESW1} have shown that if $\Delta_n$ is cube-free, then 
$K_n$ has an integral basis of the form $\{  1, \rho, \rho^2, *, * \}$. 
In this section, we give an integral basis of $K_n$ in the case that $K_n/\Q$ is tamely ramified (that is, $5\nmid n$),
which we use to give an NIB in \S~\ref{sec:NIB}. We assume that $5 \nmid n$, and hence,
we have $5 \nmid \Delta_n$.
We write $\Delta_n= ab^2c^3d^4e^5$ with  $a,b,c,d,e \in \Z_{>0}$, where $a,b,c,d$ are square-free and pairwise coprime.
From (\ref{eq:cond}), the conductor of $K_n$ is given by $\mathfrak f_{K_n} =abcd$.
\begin{theo}\label{theo:IB}  
Assume   that $5\nmid n$.
Let $u$ and $t$ be integers satisfying $5u \equiv 1 \pmod{\Delta_n}$ and 
\begin{align}
t\equiv  & (n^2+3n+4)(11n^5+110n^4+440n^3+903n^2+940n+390)\Delta_n  \notag \\
&  -un^2(11n^3+55n^2+110n+199)\delta_n^2 \pmod{\Delta_n \delta_n^2}, \label{eq:t}
\end{align}
and define an algebraic integer $T$ by 
\begin{align}
T  :=& \dfrac{f_n (\rho)-f_n (t)}{\rho-t} \notag\\
=&(t^4+t^3 \rho +t^2 \rho^2 +t \rho^3+\rho^4) +n^2 (t^3 +t^2 \rho +t \rho^2+\rho^3)  \notag \\
& -(2n^3+6n^2+10n+10)(t^2+t\rho +\rho^2) +(n^4+5n^3+11n^2+15n+5)(t+\rho) \label{eq:T} \\
& +(n^3+4n^2+10n+10). \notag 
\end{align}
Put
\[
\phi_1  := \frac{1}{e} (\rho -t), \quad 
\phi_2  := \frac{1}{cde^2} (\rho -t)^2, \quad 
\phi_3  :=\frac{1}{bcd^2e^3} (\rho -t)^3, \quad 
\phi_4  := \frac{1}{bc^2d^3e^4 \delta_n } T.
\]
Then $\{1,\phi_1,\phi_2,\phi_3,\phi_4\}$ is an integral basis of $K_n =\Q (\rho_n)$.
\end{theo}
We prove some lemmas for the proof of  Theorem~\ref{theo:IB}. Fix an integer $n$ and put $\rho =\rho_n$.
We can write  Gal$(K_n/\Q) =\langle \sigma \rangle$ where the generator $\sigma$ satisfies
\[
\sigma (\rho) =\frac{n+2+n\rho -\rho^2}{1+(n+2)\rho} 
\]
(\cite[(3.2)]{SW}). For any $x\in K_n$ and $i\in \Z/5\Z$, we write
$x^{(i)} =\sigma^i (x)$.
In what follows in this paper, we use the explicit  expression obtained by Spearman and Williams
\cite[Proposition]{SpWi}
for the conjugates of $\rho$ ($\rho^{(i)}$ corresponds to $y_i$  in \cite[Proposition]{SpWi}).

By direct calculation, we obtain  Lemmas~\ref{lem:powers}, \ref{lem:half-D} and \ref{lem:t-mod}).
\begin{lem}\label{lem:powers}
\begin{align*}
\rho^2 =& -\frac{1}{n^2} \left\{ ( 4+4 n+3 n^2+2 n^3+n^4)\rho+
2 (2+2 n+n^2) \rho^{(1)}+(2+n) (2+n+n^2) \rho^{(2)}  \right. \\
 & \left. +(2+n)^2 \rho^{(3)} +(2+n) (2+n+n^2) \rho^{(4)} \right\} ,\\
%& \\
\rho^3 =&  -\frac{1}{n^2} \left\{  
-(3+3 n+n^2) (1+2 n+4 n^2+n^3+n^4)  \rho
-(1+n) (3+6 n+4 n^2+2 n^3) \rho^{(1)} \right. \\
& -(1+n)^2 (3+3 n+2 n^2+n^3) \rho^{(2)} 
-(1+n)^2 (3+3 n+n^2) \rho^{(3)} \\
& \left. - (3+9 n+13 n^2+9 n^3+4 n^4+n^5)  \rho^{(4)} \right\}, \\
%& \\
\rho^4 =&  -\frac{1}{n^2} \left\{  
(32+72 n+99 n^2+105 n^3+86 n^4 +50 n^5+21 n^6+6 n^7+n^8 )\rho \right. \\
&\, +(32+72 n+83 n^2+60 n^3+30 n^4 +10 n^5+2 n^6 )\rho^{(1)} \\
&\, +(32+72 n+94 n^2+80 n^3+47 n^4 +20 n^5+6 n^6+n^7 ) \rho^{(2)} \\
&\, +(32+72 n+75 n^2+51 n^3+24 n^4  +7 n^5+n^6 ) \rho^{(3)} \\
&\left. +(32+72 n+93 n^2+80 n^3+49 n^4 +21 n^5+6 n^6+n^7 ) \rho^{(4)} \right\} .
\end{align*}
\end{lem}
\begin{lem}\label{lem:half-D}
\begin{enumerate}
\item[(1)] $\displaystyle{ \prod_{i=0}^4 (\rho^{(i)} -\rho^{(i+1)} )=\delta_n \Delta_n}$.
\item[(2)] $\displaystyle{ \prod_{i=0}^4 (\rho^{(i)} -\rho^{(i+2)} )=-\Delta_n }$.
\end{enumerate}
\end{lem}
\begin{lem}\label{lem:t-mod}
Assume that $5\nmid n$.
Let $u$ and $t$  be the integers in Theorem~\ref{theo:IB}. We have
$t\equiv -un^2 \pmod{\Delta_n}$ and $t\equiv -(n^2+3n+4) \pmod{\delta_n^2}$.
\end{lem}
We put
\begin{equation}\label{eq:abcd-prod}
ab^2c^3d^4=\prod_{k=1}^{\ell }p_k^{m_k}, \quad \Delta_n =e^5 \prod_{k=1}^{\ell }p_k^{m_k}
\end{equation}
with $p_1,\ldots,p_{\ell}$ being distinct prime numbers and $m_k \in \{1,2,3,4 \}$.
For each $ k$, let $\mathfrak p_k$ be a prime ideal of $K_n$ dividing $p_k$. Since $p_k$ is ramified in $K_n$, we have $p_k \mathcal O_{K_n} =(p_k)=\mathfrak p_k^5$.
Let $v_{\mathfrak p}(x)$ be the $\mathfrak p$-adic valuation of $x \in K_n$ for a prime ideal $\mathfrak p$.
\begin{lem}\label{lem:r-r}
For any $i\in \Z/5\Z$, we have $(\rho^{(i)} -\rho^{(i+2)} )=e \prod_{k=1}^{\ell }\mathfrak p _k^{m_k}$, and hence
$(\rho^{(i)}-\rho^{(i+2)})/(\rho^{(j)}-\rho^{(j+2)}) \in \mathcal O_{K_n}^{\times} $ for any $i,j\in \Z/5\Z$.
\end{lem}
\begin{proof}\
It follows from Lemma~\ref{lem:half-D} (2) that $v_{{\mathfrak p}_k}(\rho^{(0)}-\rho^{(2)})=v_{{\mathfrak p}_k}(\rho^{(1)}-\rho^{(3)})=
v_{{\mathfrak p}_k}(\rho^{(2)}-\rho^{(4)})= v_{{\mathfrak p}_k}(\rho^{(3)}-\rho^{(0)})=v_{{\mathfrak p}_k}(\rho^{(4)}-\rho^{(1)})
$  for $k=1,\ldots \ell$ because these primes $p_1,\ldots,p_{\ell}$ are ramified in $K_n$, and hence
${\mathfrak p}_k^{m_k} || (\rho^{(i)}-\rho^{(i+1)})$ for any $i\in \mathbb Z/5\mathbb Z$.
Let $p$ be a prime number dividing $e$ but different from  $p_1,\ldots, p_{\ell}$, and $p^m || e \ (m\in \mathbb Z_{>0})$.
The minimal polynomial of $\rho-\rho^{(1)}$ is given by
\[
G_n(X):=X^5-\Delta_n X^3+(n+1)\Delta_n X^2+(n^2+4n+5)\Delta_n X-\delta_n \Delta_n
\]
(see \cite[p.78]{Je}). Let
\begin{align*}
g_n(X) & =p^{-5m} G_n (p^m X) \\
& =X^5-\Delta_n p^{-2m} X^3+(n+1)\Delta_n p^{-3m} X^2+ (n^2+4n+5)\Delta_n p^{-4m} X-\delta_n \Delta_n p^{-5m}.
\end{align*}
Then we have $g_n (X) \in \mathbb Z[X]$ and $g_n ((\rho-\rho^{(1)} )/p^m )=0$, and hence $(\rho-\rho^{(1)})/p^m \in O_{K_n}$.
By acting elements of $\mathrm{Gal} (K_n/\mathbb Q)$, we obtain $(\rho^{(i)}-\rho^{(i+1)})/p^m \in O_{K_n}$ for any
$i\in \mathbb Z/5\mathbb Z$, and $(\rho^{(i)}-\rho^{(i+2)})/p^m =(\rho^{(i)}-\rho^{(i+1)} )/p^m +(\rho^{(i+1)} -\rho^{(i+2)})/p^m
\in O_{K_n}$. 
Furthermore, from Lemma~\ref{lem:half-D} (2) and (\ref{eq:abcd-prod}), we have $p^m || (\rho^{(i)}-\rho^{(i+1)})$. The claim follows from these facts.
%\qed
\end{proof}
\begin{lem}\label{lem:r-integral}
Assume that $5\nmid n$.
Let $t$  be the integer and $T$ be the element of $\mathcal O_{K_n}$  in Theorem~\ref{theo:IB}.  We have 
\[
\frac{\rho-t}{\rho -\rho^{(2)}}\  \in \mathcal O_{K_n}, \quad \frac{T}{(\rho-\rho^{(2)})^4 \delta_n} \  \in \mathcal O_{K_n}.
\]
\end{lem}
\begin{proof}\
Put $y:=\rho -\rho^{(2)},\ \phi:=(\rho-t)/y$ and $\xi:=T/(y^4\delta_n) $. First, we show that  $\phi \in \mathcal O_{K_n}$.
Substituting $\rho=y \phi +t$ for $X$ in $f_n(X)=\sum_{k=0}^5 \frac{f_n^{(k)}(t)}{k!} (X-t)^k$, we get 
$\sum_{k=0}^5 \frac{f_n^{(k)} (t)}{k!} (y\phi )^k=0$ and hence $\phi$ is a root of the monic polynomial
\begin{equation}\label{eq:phi-poly}
\sum_{k=0}^5 \frac{f_n^{(k)}(t)}{k! y^{5-k}} X^k.
\end{equation}
We show that all coefficients of this polynomial  are algebraic integers.
By direct calculation, we have
$f_n^{(k)}(-un^2)/k! \equiv 0 \pmod{\Delta_n}$
for any $k\in \{ 0,\ldots, 5\}$. Since $(\Delta_n )=(y^5)$ from Lemmas~\ref{lem:half-D} and \ref{lem:r-r}, and $t\equiv -un^2 \pmod{\Delta_n}$ from Lemma~\ref{lem:t-mod},
we have
\begin{equation}\label{eq:phi-integral}
\frac{f_n^{(k)}(t)}{k!} \in (\Delta_n )=(y^5) \quad (k\in \{0,\ldots, 5\})
\end{equation}
and hence all the coefficients of the polynomial  (\ref{eq:phi-poly}) are algebraic integers, which implies $\phi \in \mathcal O_{K_n}$.

Next, we show that $\xi \in \mathcal O_{K_n}$. Since $\rho$ is a root of $f_n(X)$, we have $-f_n(t)=f_n (\rho)-f_n(t)=(\rho-t)T$ and hence
\begin{equation}\label{eq:r-t}
\rho -t =-\frac{f_n(t)}{T} =-\frac{f_n(t)}{y^4\delta_n \xi}.
\end{equation}
Furthermore, substituting $\rho$ for $X$ in $f_n(X)=\sum_{k=0}^5 \frac{f_n^{(k)}(t)}{k!} (X-t)^k$, we obtain
\begin{equation}\label{eq:sum=0}
\sum_{k=0}^5 \frac{f_n^{(k)} (t) }{k!} (\rho -t)^k=0.
\end{equation}
From (\ref{eq:r-t}) and (\ref{eq:sum=0}), we obtain
\[
\sum_{k=0}^5 \frac{ f_n^{(k)}(t) f_n (t)^{k-1}}{k!} \left(- \frac{1}{y^4 \delta_n } \right)^k \xi^{5-k} =0.
\]
Therefore, $\xi$ is a root of the monic polynomial 
\begin{equation}\label{eq:xi-poly}
\sum_{k=0}^5 \frac{ f_n^{(k)}(t) f_n (t)^{k-1}}{k!} \left(- \frac{1}{y^4 \delta_n } \right)^k X^{5-k}.
\end{equation}
We show that all coefficients of this polynomial  are algebraic integers.
By direct calculation, we obtain $f_n(-(n^2+3n+4)) \equiv 0 \pmod{\delta_n^2}$ and $f_n'(-(n^2+3n+4)) \equiv 0
\pmod{\delta_n}$. Since $t\equiv -(n^2+3n+4) \pmod{\delta_n^2}$ from Lemma~\ref{lem:t-mod}, we have
\begin{equation}\label{eq:delta-delta}
f_n(t) \in (\delta_n^2), \quad f_n'(t) \in (\delta_n).
\end{equation}
From (\ref{eq:phi-integral}), (\ref{eq:delta-delta}) and $(\Delta_n ,\delta_n )=1$, we obtain $f_n (t) \in (\Delta_n \delta_n^2) =(y^5 \delta_n^2)$ 
and $f_n'(t) \in (\Delta_n \delta_n )=(y^5 \delta_n)$, and hence
\begin{equation}\label{eq:f-f'}
\frac{f_n(t)}{y^5 \delta_n^2}\ \in \mathcal O_{K_n} , \quad \frac{f_n'(t)}{y^5\delta_n} \ \in \mathcal O_{K_n}.
\end{equation}
We see all  coefficients of the polynomial  (\ref{eq:xi-poly}) are algebraic integers from  (\ref{eq:phi-integral}) and (\ref{eq:f-f'}), which implies $\xi\in \mathcal O_{K_n}$.
%\qed
\end{proof}

We give the proof of Theorem~\ref{theo:IB}. For $x_1,\ldots, x_5 \in K_n$, we define the discriminant
\[
d(x_1,\ldots, x_5) := \begin{vmatrix}
x_1 & x_2 & \cdots & x_5 \\
x_1^{(1)} & x_2^{(1)} & \cdots & x_5^{(1)} \\
\vdots & \vdots &  & \vdots \\
x_1^{(4)}  & x_2^{(4)}  &  \cdots & x_5^{(4)} 
\end{vmatrix}^2.
\]
The set $ \{ x_1,\ldots, x_5 \}$ is an integral basis of $K_n$ if and only if $x_1,\ldots, x_5 \in \mathcal O_{K_n}$ and $d(x_1,\ldots ,x_5)=D_{K_n}$.
\begin{proof}[Proof of Theorem~\ref{theo:IB}]\ 
From the multilinearity of a determinant, it follows that 
\[
d(1,\phi_1,\phi_2,\phi_3,\phi_4) =\frac{1}{(b^2c^4d^6e^{10} \delta_n)^2 }{d(1,\rho,\rho^2,\rho^3,\rho^4)} .
\]
Since $d(1,\rho, \rho^2,\rho^3,\rho^4)=d(f_n) =\delta_n^2 \Delta_n^4$, we obtain 
\[
d(1,\phi_1,\phi_2,\phi_3,\phi_4) = \mathfrak f_{K_n}^4=D_{K_n}.
\]

Next, we show that $\phi_1,\phi_2,\phi_3, \phi_4 \in \mathcal O_{K_n}$. 
From Lemmas \ref{lem:r-r} and \ref{lem:r-integral}, we have
\begin{equation}\label{eq:(r-t)}
(\rho -t) \subset (\rho -\rho^{(2)} ) =e \prod_{k=1}^{\ell} \mathfrak p_k^{m_k},
\end{equation}
and hence $\phi_1 =(\rho -t)/e \in \mathcal O_{K_n}$.  Let $p_j$ be a prime number dividing $cd$, then we have $m_j \in \{ 3,4 \}$.
From (\ref{eq:(r-t)}), we obtain
$( (\rho -t)/e)^2 \subset \prod_{k=1}^{\ell} \mathfrak p_k^{2m_k}$.
Since $c$ and $d$ are square-free and coprime, we have
\[
v_{\mathfrak p_j} (cd)=5 \leq 6 \leq v_{\mathfrak p_j} ( \prod_{k=1}^{\ell} \mathfrak p_k^{2m_k} ) \leq v_{\mathfrak p_j} ( ( \frac{\rho -t}{e} )^2 ).
\]
Since this inequality holds for all primes $p_j$  dividing $cd$, we conclude that $\phi_2= (\rho-t)^2/cde^2 \in \mathcal O_{K_n}$. 
Similarly,  let $p_j$ be a prime number dividing $bcd$, then we have $m_j \in \{ 2,3,4 \}$.
From (\ref{eq:(r-t)}), we obtain
$( (\rho -t)/e)^3 \subset \prod_{k=1}^{\ell} \mathfrak p_k^{3m_k}$.
Since $b,c$ and $d$ are square-free and pairwise coprime, if $p_j$ divides $bc$, then we have
\[
v_{\mathfrak p_j} (bcd^2)=5 \leq 6 \leq v_{\mathfrak p_j} ( \prod_{k=1}^{\ell} \mathfrak p_k^{3m_k} ) \leq v_{\mathfrak p_j} ( ( \frac{\rho -t}{e} )^3 ), 
\]
and if $p_j$ divides $d$, then we have
\[
v_{\mathfrak p_j} (bcd^2)=10 \leq 12 = v_{\mathfrak p_j} ( \prod_{k=1}^{\ell} \mathfrak p_k^{3m_k} ) \leq v_{\mathfrak p_j} ( ( \frac{\rho -t}{e} )^3 ).
\]
We conclude that $\phi_3= (\rho-t)^3/bcd^2e^3 \in \mathcal O_{K_n}$. 
Finally, we show that $\phi_4 \in \mathcal O_{K_n}$.
From Lemmas~\ref{lem:r-r} and \ref{lem:r-integral}, we have
\[
 \left( \frac{T}{\delta_n} \right) \subset (\rho-\rho^{(2)} )^4=e^4 \prod_{k=1}^{\ell} \mathfrak p_k^{4m_k},
\]
and hence 
\[
 \left( \frac{T}{e^4 \delta_n}  \right) \subset \prod_{k=1}^{\ell} \mathfrak p_k^{4m_k}.
\]
Let $p_j$ be a prime number dividing $bcd$. If $p_j$ divides $b$, then we have $m_j=2$ and
\[
v_{\mathfrak p_j} (bc^2d^3)=5 \leq 8 = v_{\mathfrak p_j} ( \prod_{k=1}^{\ell} \mathfrak p_k^{4m_k} ) \leq v_{\mathfrak p_j} ( \frac{T}{e^4 \delta_n} ),
\]
and if $p_j$ divides $c$, then we have $m_j=3$ and 
\[
v_{\mathfrak p_j} (bc^2d^3)=10 \leq 12 = v_{\mathfrak p_j} ( \prod_{k=1}^{\ell} \mathfrak p_k^{4m_k} ) \leq v_{\mathfrak p_j} ( \frac{T}{e^4\delta_n}).
\]
If $p_j$ divides $d$, then we have $m_j=4$ and
\[
v_{\mathfrak p_j} (bc^2d^3)=15 \leq 16 = v_{\mathfrak p_j} ( \prod_{k=1}^{\ell} \mathfrak p_k^{4m_k} ) \leq v_{\mathfrak p_j} ( \frac{T}{e^4\delta_n}).
\]
Since these inequalities hold for all primes $p_j$  dividing $bcd$,
we conclude that $\phi_4= T/(bc^2d^3e^4 $ $\delta_n ) \in \mathcal O_{K_n}$. 
The proof is complete.
%\qed
\end{proof}

The following example is a case that $\Delta_n$ is cube-free, and Eloff, Spearman and Williams
give an integral basis in their paper \cite[p.~770]{ESW1}.
\begin{ex}\label{ex:n=14}
Let $n=14$. We have $\Delta_n=11\times 71^2,\  a=11,\  b=71,\ c=d=e=1,\ \delta_n=79\times 7^2 ,\, \mathfrak f_{K_n}=11\times 71$ and $D_{K_n}=11^4 \times 71^4$.
Integers $u=44361$ and $t=645583287961$ satisfy $5u \equiv 1  \pmod{\Delta_n}$ and (\ref{eq:t})  respectively.
Put
\[
\phi_1 := \rho-t,\ \phi_2:=(\rho-t)^2, \ \phi_3:= \frac{1}{71} (\rho-t)^3,\ \phi_4:= \frac{1}{71 \delta_n} T
\]
with 
\begin{align*}
T:= &\rho^4 + 645583288157\rho^3 + 416777781821069771971063\rho^2  \\
   & +  269064770737138517575467371288327050\rho \\
   & + 173703719366954541829067611507591745478095648728.
\end{align*}
It follows that $\{ 1, \phi_1, \phi_2, \phi_3, \phi_4 \}$ is an integral basis of $K_n$ from Theorem~\ref{theo:IB}.
On the other hand, Eloff, Spearman and Williams \cite[Example]{ESW1} gave another integral basis $\{ 1,\rho,\rho^2, v_4,v_5 \}$ with
\begin{align*}
v_4 & := \frac{1}{71} (5+29 \rho +4\rho^2+\rho^3) , \\
v_5 & := \frac{1}{274841}(50339+27624\rho+112706 \rho^2+220601 \rho^3+\rho^4).
\end{align*}
The two bases satisfy the relation $(1,\rho,\rho^2,v_4,v_5)R=(1,\phi_1, \phi_2,\phi_3,\phi_4)$  by the matrix $R$ with det$(R)$=1 given by
\[
R:=  \left[ \begin{array}{ccccc}
1 & a_{12} & a_{13} & a_{14} & a_{15} \\
0 & 1  & a_{23} & a_{24} & a_{25} \\
0  & 0  & 1  & a_{34} & a_{35} \\
0 & 0 &  0  & 1 & a_{45}\\
0 & 0 & 0 & 0 & 1
\end{array} \right], 
\]
where
{\small
\begin{align*}
a_{12} &=-645583287961,\\
a_{13} &= 416777781694535447537521,\\
a_{14} &=-3789644657119015103770949691480066, \\
a_{15}&=632015308367217925378919489841733010177449, \\
a_{23} &=-1291166575922, \\
a_{24}&=17610328803994455529754,\\
 a_{25}&=978983378524814411152079381822,\\
a_{34}&=-27278167097,\\
a_{35}&= 1516432343858767213,\\
a_{45}&=166774236.
\end{align*}
}
\end{ex}

The next example is a case that $\Delta_n$ is not cube-free  and no NIB has been obtained in  \cite[p.~771]{ESW1}.
\begin{ex}\label{ex:n=44}
Let $n=44$. We have $\Delta_n=61\times 41^3,\  a=61,\  b=1,\ c=41,\ d=e=1,\ \delta_n=95311 $ and $\mathfrak f_{K_n}=41\times 61$ and $D_{K_n}=41^4 \times 61^4$.
Integers $u=3363345$ and $t=30447786579308863$ satisfy $5u \equiv 1  \pmod{\Delta_n}$ and (\ref{eq:t})  respectively.
Put
\[
\phi_1 := \rho-t,\ \phi_2:=\frac{1}{41} (\rho-t)^2, \ \phi_3:= \frac{1}{41} (\rho-t)^3,\ \phi_4:= \frac{1}{41^2 \delta_n} T
\]
with 
\begin{align*}
T:= & \rho^4 + 30447786579310799\rho^3 + 927067707579199859568212292129103\rho^2\\
    & +28227159704940594977085733369537146411892512335866 \rho \\
    &  +859454534436098172963532872007773204827694194742789092180006673736.
\end{align*}
It follows that $\{ 1, \phi_1, \phi_2, \phi_3, \phi_4 \}$ is an integral basis of $K_n$ from Theorem~\ref{theo:IB}.
\end{ex}
\section{Normal integral bases}\label{sec:NIB}
Let $n$ be an integer and put $\zeta:=\zeta_5$.
We construct the generator of an $\mathrm{NIB}$ from a normal basis and the
integral basis given by Theorem\,\ref{theo:IB}.
The idea is based on the method  of Acciaro and Fieker \cite{AF}.
\begin{lem}\label{lem:NB}
The set $\{ \rho, \rho^{(1)}, \rho^{(2)}, \rho^{(3)},\rho^{(4)} \}$ is a normal basis if and only if $n\ne 0$.
\end{lem}
\begin{proof}\
Since $K_n$ and $\Q (\zeta)$ are linearly disjoint over $\Q$,  it follows that  $1, \zeta^k, \zeta^{2k}, \zeta^{3k}, \zeta^{4k} $ are 
linearly independent over $K_n$ for any   $k\in \{1,2,3,4 \}$, and hence 
\[
\rho +\rho^{(1)} \zeta^k +\rho^{(2)} \zeta^{2k} +\rho^{(3)} \zeta^{3k} +\rho^{(4)} \zeta^{4k} \ne 0.
\]
Since $d(\rho,\rho^{(1)},\rho^{(2)},\rho^{(3)},\rho^{(4)})$ is the square of the determinant of the circulant matrix,
we conclude that
\[
d(\rho, \rho^{(1)} ,\rho^{(2)}, \rho^{(3)},\rho^{(4)} )=\prod_{k=0}^4 ( \rho +\rho^{(1)} \zeta^k +\rho^{(2)} \zeta^{2k} +\rho^{(3)} \zeta^{3k} +\rho^{(4)} \zeta^{4k})^2 \ne 0
\]
if and only if
\[
\rho+\rho^{(1)}+\rho^{(2)} +\rho^{(3)} +\rho^{(4)} =-n^2 \ne 0.
\]
%\qed
\end{proof}

Put $A_n:=n+2+2\zeta^4+\zeta^2,\ B_n:=n+2+2\zeta^2+\zeta,\ C_n:=n+2+2\zeta^3+\zeta^4,\ D_n:=n+2+2\zeta+\zeta^3$. We have $\Delta_n=A_nB_nC_nD_n=N_{\Q(\zeta)/\Q}(A_n)$. For $\lambda \in \Z [\zeta]$ satisfying $\lambda \not\in (1-\zeta)$, we have $\lambda \equiv \pm1, \pm2 \pmod{(1-\zeta)}$. We define $u_{\lambda} \in \Z [\zeta ]^{\times}$ by
\[
u_{\lambda}=\begin{cases}
1, & \text{if $\lambda \equiv 1 \pmod{(1-\zeta)} $}, \\
-1, & \text{if $\lambda \equiv -1 \pmod{(1-\zeta)} $}, \\
\frac{1+\sqrt{5}}{2}, & \text{if $\lambda \equiv 2 \pmod{(1-\zeta)} $}, \\
-\frac{1+\sqrt{5}}{2}, & \text{if $\lambda \equiv -2 \pmod{(1-\zeta)} $}.
\end{cases}
\]
\begin{lem}\label{lem:prime-ideal}
If  $\mathfrak p$ is  a prime ideal of $\Q(\zeta)$ including at least two of $A_n, B_n, C_n, D_n$, then we have $\mathfrak p =(1-\zeta)$ $($the prime ideal above $5)$.
In particular, if $5 \nmid \Delta_n$ then $A_n, B_n, C_n,$ and $D_n$ are pairwise coprime.
\end{lem}
\begin{proof}\
Assume that $A_n, B_n \in \mathfrak p$ (we can show other cases similarly).
Since $A_n-B_n=2\zeta^4-\zeta^2-\zeta \in \mathfrak p$ and 
$N_{\Q(\zeta)/\Q}(2\zeta^4-\zeta^2-\zeta)=5^2$,  we have $\mathfrak p=(1-\zeta)$.
%\qed
\end{proof}
\begin{lem}\label{lem:p=0,1} \cite[Lemma~2.1.1]{Je} 
Let $p$ be a prime number dividing $\Delta_n$. We have $p\equiv 0,1 \pmod{5}$.
\end{lem}
\begin{lem}\label{lem:alpha-A}
Assume  $5\nmid n$.
Let $s=p_1 \cdots p_k$ ($p_1,\cdots,p_k$ are not necessarily different prime numbers) 
be a positive integer dividing $\Delta_n$.  Then, $p_i \, (i\in  \{1,\ldots, k\})$ 
 decomposes into a product of prime elements $\pi_i^{(t)} \, (t \in \{1,\ldots 4 \})$  in $\Z [\zeta]$ as follows.
 \begin{equation}\label{eq:alpha-A}
 p_i= \prod_{t=1}^4 \pi_i^{(t)}, \quad \pi_i^{(1)}|A_n ,\ \pi_i^{(2)} |B_n,\ \pi_i^{(3)} |C_n,\ \pi_i^{(4)} |D_n \ \text{ in } \Z [\zeta].
 \end{equation}
 Furthermore,  put $\lambda_t := \prod_{i=1}^k \pi_i^{(t)} \, (\in \Z [\zeta ])$, and let $u_{\lambda_t}$ be the unit defined before Lemma~\ref{lem:prime-ideal}
 and $\alpha_A:=u_{\lambda_1} \lambda_1,\ \alpha_B:=u_{\lambda_2} \lambda_2,\ \alpha_C:=u_{\lambda_3} \lambda_3,\ \alpha_D:=u_{\lambda_4} \lambda_4$
 then we have  $s=N_{\Q (\zeta)/\Q} (\alpha_A)=N_{\Q (\zeta)/\Q} (\alpha_B)=N_{\Q (\zeta)/\Q} (\alpha_C)=N_{\Q (\zeta)/\Q} (\alpha_D)$,
 $\alpha_A \equiv \alpha_B \equiv \alpha_C\equiv \alpha_D\equiv 1 \pmod{(1-\zeta)}$ and  $\alpha_A |A_n,\ \alpha_B | B_n,\ \alpha_C | C_n ,\ \alpha_D |D_n$
 in $\Z[\zeta]$.
\end{lem}
\begin{proof}\
From Lemma\,\ref{lem:p=0,1} and $5\nmid n$, we have $p_1 \equiv \cdots \equiv p_k \equiv 1 \pmod{5}$.
Since $\mathbb{Z}[\zeta]$ is a unique factorization domain and
$p_1,\cdots,p_k$ split completely in $\mathbb{Q}(\zeta)$,
it follows that 
$p_i \, (i\in  \{1,\ldots, k\})$ 
 decomposes as (\ref{eq:alpha-A}). 
 Since $u_{\lambda_t} \in \mathbb Z [\zeta]^{\times}$, we have 
 \[
 N_{\mathbb Q (\zeta)/\mathbb Q} (u_{\lambda_t} \lambda_t)=N_{\mathbb Q (\zeta)/\mathbb Q}(u_{\lambda_t}) \prod_{i=1}^k N_{\mathbb Q (\zeta)/\mathbb Q} (\pi_i^{(t)})=
 \prod_{i=1}^kp_i=s
 \]
 for any $t\in \{1,\ldots, 4\}$. The claim $\alpha_A \equiv \alpha_B\equiv \alpha_C \equiv \alpha_D \equiv 1 \pmod{(1-\zeta)}$
 follows from the definition of $u_{\lambda_t}$, and $\alpha_A |A_n,\ \alpha_B | B_n,\ \alpha_C | C_n ,\ \alpha_D |D_n$ follows from 
Lemma~\ref{lem:prime-ideal}. %\qed
\end{proof}

Assume  $5\nmid n$ and let $ \{1,\phi_1,\phi_2,\phi_3,\phi_4 \}$ be the integral basis given by Theorem~\ref{theo:IB}.
We find the generator of an NIB based on the method of Acciaro and Fieker \cite{AF}. Before going into a detailed discussion, we
briefly explain their method. First, we see $\ell \mathcal O_{K_n} \subset \mathbb Z [G] \cdot \rho$ for $\ell :=-bc^2d^3e^4\delta_n n^2$ in 
Lemma~\ref{lem:phi-to-rho}.  In fact, we will find elements $g_0,\ldots, g_4 \in \mathbb Z [G]$ satisfying 
$1=g_0 \cdot  \rho/\ell, \ \phi_1 =g_1 \cdot \rho/\ell,\ \phi_2=g_2 \cdot  \rho/\ell,\
\phi_3=g_3 \cdot  \rho/\ell, \ \phi_4=g_4 \cdot  \rho/\ell$. On the other hand, for the generator $\alpha$ of an NIB of $K_n$,
we can write $\alpha=g \cdot \rho/\ell$ for some $g\in \mathbb Z [G]$.
By finding the element $g$, we can find $\alpha$. We will show  (\ref{eq:g01234}) that the principal ideal $(g)$ in the group ring $\mathbb Z[G]$
is equal to the ideal generated by $g_0,\ldots, g_4$. To find the generator of the
ideal $I:=(g_0,\ldots, g_4)$, we transfer it to a unique factorization domain $\mathbb Z[\zeta]$ by the surjection $\nu :\mathbb Z [G]
\to \mathbb Z [\zeta],\ \sigma \mapsto \zeta$, and pull the generator of the ideal $\nu (I)$ back to the group ring. And 
finally we can find $g$.

By direct calculation, we obtain the following lemma.
\begin{lem}\label{lem:phi-to-rho}
Assume  $5\nmid n$ and put $\ell :=-bc^2d^3e^4 \delta_n n^2$. We have
\[
1=g_0 \cdot \frac{\rho}{\ell}, \ \phi_1 =g_1 \cdot \frac{\rho}{\ell},\ \phi_2=g_2 \cdot  \frac{\rho}{\ell},\
\phi_3=g_3 \cdot \frac{\rho}{\ell}, \ \phi_4=g_4 \cdot \frac{\rho}{\ell}
\]
with $g_0,g_1,g_2,g_3,g_4 \in \Z [G]$ given by
\begin{align*}
g_0  := &-\dfrac{\ell}{n^2} (1+\sigma+\sigma^2+\sigma^3+\sigma^4), \\
g_1 := & \dfrac{\ell}{en^2} \{ (n^2+t)+t\sigma +t\sigma^2+t\sigma^3+t\sigma^4 \}, \\
g_2  := & -\dfrac{\ell}{cde^2n^2}  \{(4+4n+3n^2+2n^3+n^4+2n^2t+t^2) +
(4+4n+2n^2+t^2)\sigma  \\
& +(4+4n+3n^2+n^3+t^2)\sigma^2  +(4+4n+n^2+t^2)\sigma^3 \\
&+(4+4n+3n^2+n^3+t^2)\sigma^4
\},
\end{align*}
\begin{align*}
g_3 := & -\dfrac{\ell}{bcd^2e^3n^2} \{  
(-3-9n-19n^2-17n^3-10n^4-4n^5-n^6-12t-12nt \\
& -9n^2t-6n^3t-3n^4t-3n^2t^2-t^3)\\
&+(-3-9n-10n^2-6n^3-2n^4-12t-12nt-6n^2t-t^3)\sigma \\
& +(-3-9n-11n^2-8n^3-4n^4-n^5-12t-12nt-9n^2t-3n^3t-t^3)\sigma^2 \\
&+(-3-9n-10n^2-5n^3-n^4-12t-12nt-3n^2t-t^3)\sigma^3 \\
&+(-3-9n-13n^2-9n^3-4n^4-n^5-12t-12nt-9n^2t-3n^3t-t^3)\sigma^4
\}, \\
g_4 := & -\dfrac{\ell}{bc^2d^3e^4\delta_n n^2} \{  
(2+2n+n^2+2t+6nt+6n^2t+2n^3t-6t^2-6nt^2-3n^2t^2+t^4) \\
&+(2+2n+2t+6nt+5n^2t+3n^3t+n^4t-6t^2-6nt^2-4n^2t^2-2n^3t^2+n^2t^3+t^4)\sigma \\
& +(2+2n+n^2+2t+6nt+4n^2t+n^3t-6t^2-6nt^2-3n^2t^2-n^3t^2+n^2t^3+t^4)\sigma^2 \\
& +(2+2n+2n^2+n^3+2t+6nt+5n^2t+4n^3t+n^4t \\
& -6t^2-6nt^2-5n^2t^2-2n^3t^2+n^2t^3+t^4)\sigma^3 \\
&+ (2+2n+2t+6nt+2n^2t-6t^2-6nt^2-3n^2t^2-n^3t^2+n^2t^3+t^4)\sigma^4
\}.
\end{align*}
\end{lem}

From Lemma~\ref{lem:phi-to-rho} it follows that 
\begin{align} 
\mathcal{O}_{K_n} &
=\Z +\phi_1 \Z+\phi_2 \Z+\phi_3\Z+\phi_4\Z  \notag \\
&= (g_0 \Z [G]+g_1 \Z [G] +g_2\Z[G] +g_3\Z[G] +g_4\Z[G] )\cdot  \dfrac{\rho}{\ell}  \label{eq:O_Ln}
\end{align}
Let $\alpha$ be a generator of an NIB of $K_n$, then there exists $g \in \Z [G]$ satisfying 
\begin{equation}\label{eq:alpha-g}
\alpha=g \cdot \frac{\rho}{\ell}.
\end{equation}
Therefore, we have 
\begin{equation}\label{eq:O_Ln-g}
\mathcal O_{K_n }=\Z [G] \cdot \alpha =g\Z [G]\cdot \frac{\rho}{\ell}.
\end{equation}
From (\ref{eq:O_Ln}) and  (\ref{eq:O_Ln-g}),
we obtain the equality as ideals of $\mathbb{Z}[G]$ :
\[
(g)+\mathrm{Ann}_{\mathbb{Z}[G]}\left(\frac{\rho}{\ell}\right)
=(g_0, g_1,g_2,g_3,g_4)+\mathrm{Ann}_{\mathbb{Z}[G]}\left(\frac{\rho}{\ell}\right).
\]
Since $\{\rho/\ell ,\rho^{(1)}/\ell ,\rho_n^{(2)}/\ell, \rho^{(3)} /\ell, \rho^{(4)}/\ell \}$ is a normal basis of $K_n$
from Lemma~\ref{lem:NB},
we have $\mathrm{Ann}_{\mathbb{Z}[G]}\left(\rho/\ell \right)=0$,
and hence we get the equality as ideals of $\mathbb{Z}[G]$ :
\begin{equation}\label{eq:g01234}
(g)=(g_0,g_1,g_2,g_3,g_4).
\end{equation}
Consider the surjective ring homomorphism
\[
\nu\ :\ \mathbb{Z}[G]\longrightarrow\mathbb{Z}[\zeta],
\]
defined by $\nu(\sigma)=\zeta$.
We calculate the image of the ideal $I:=(g)=(g_0,g_1,g_2,g_3,g_4)$
by $\nu$.
Since $\nu$ is surjective, we obtain the ideal of $\mathbb{Z}[\zeta]$ :
\begin{equation}\label{eq:nu-I}
\nu(I)=(\nu(g))=(\nu(g_0), \nu(g_1),\nu(g_2),\nu(g_3),\nu(g_4)).
\end{equation}
\begin{lem}\label{lem:nu(g)}
Assume that $5 \nmid n$. We have the following.
\begin{itemize}
\item[(1)] $\nu (g_0)=0$.
\item[(2)] $\nu (g_1)=- n^2 \delta_n bc^2d^3e^3$.
\item[(3)] $\nu (g_2)=n^2 \delta_n A_n bcd^2e^2 x$ with  $x\in \Z [\zeta ]$ which is prime to $B_n$.
\item[(4)] $\nu (g_3)=n^2 \delta_n A_n B_n cde y$ with  $y\in \Z [\zeta ]$ which is prime to $C_n$.
\item[(5)] $\nu (g_4)=n^2 \delta_n A_n B_nC_n z$ with  $z\in \Z [\zeta ]$ which is prime to $D_n$.
\end{itemize}
\end{lem}
\begin{proof}\ 
We can easily prove (1) and (2).
\begin{itemize}
\item[(3)] We have $\nu (g_2)=n^2 \delta_n bcd^2e^2 H_2$ with $H_2:=n^2+(1-\zeta-\zeta^3)n +2t-2\zeta^3-\zeta$.
Since $5u \equiv 1 \pmod{\Delta_n }$ and $t\equiv -un^2 \pmod{ \Delta_n }$, it follows that $
H_2 \equiv uA_n (3n+4+2\zeta^2-\zeta^4) \pmod{\Delta_n}$. Furthermore, we have $3n+4+2\zeta^2-\zeta^4 \equiv \zeta^3-3\zeta^2-2\zeta-1 \pmod{B_n}$
and $N_{\Q (\zeta)/\Q}(\zeta^3-3\zeta^2-2\zeta-1)=5^3$.  Put $x:=H_2/A_n \, (\in \Z[\zeta])$. Then it follows from $x\equiv u(
3n+4+2\zeta^2-\zeta^4) \pmod{B_n}$ that $x$ is prime to $B_n$ and  $\nu (g_2)=n^2
\delta_n A_n bcd^2e^2x$.
\item[(4)] We have $\nu (g_3)=n^2 \delta_n cde H_3$ with 
\begin{align*}
H_3 := & \{-n^4+(\zeta^3+\zeta-3)n^3+(3\zeta^3+2\zeta-6-3t)n^2+(4\zeta^3+\zeta^2+3\zeta-8 \\
& +3\zeta^3t+3\zeta t-3t)n -3t^2+6\zeta^3t+3\zeta t+3\zeta^3+2\zeta^2+3\zeta-6\}.
\end{align*}
Since $5u \equiv 1 \pmod{\Delta_n }$ and $ t\equiv -un^2 \pmod{ \Delta_n }$, it follows that $
H_3 \equiv u^2A_n B_n H_3'\pmod{\Delta_n}$ with
\[
H_3':=-13n^2+(-16\zeta^3+13\zeta^2-3\zeta-34)n-5\zeta^3+40\zeta^2+10\zeta-5.
\]
Furthermore, we have $H_3' \equiv -10\zeta^3-10\zeta^2-5\zeta\pmod{C_n}$
and $N_{\Q (\zeta)/\Q}(-10\zeta^3-10\zeta^2-5\zeta)=5^5$.  Put $y:=H_3/(A_n B_n) \, (\in \Z[\zeta])$. Then it follows from $y\equiv u^2 H_3' \pmod{C_n}$
that $y$ is prime to $C_n$ and  $\nu (g_3)=n^2
\delta_n A_n B_n cdey$.
\item[(5)] We have $\nu (g_4)=n^2 H_4$ with 
\begin{align*}
H_4 := & \{ (\zeta^3+\zeta)n^2t +(-\zeta^3-\zeta+1)nt^2+(4\zeta^3+\zeta^2+3\zeta+2)nt +\zeta^3n-t^3  \\
& + (-2\zeta^3-\zeta)t^2+(3\zeta^3+2\zeta^2+3\zeta+4)t+(2\zeta^3+\zeta^2+1)\}.
\end{align*}
Since $5u \equiv 1 \pmod{\Delta_n }$ and $t\equiv -un^2 \pmod{ \Delta_n }$, it follows that $
H_4 \equiv u^3A_n B_n C_nH_4'\pmod{\Delta_n}$ with
\[
H_4':=n^3+(-4\zeta^3-3\zeta+2)n^2+(-5\zeta^3+5\zeta^2-5\zeta)n+15\zeta^3-10\zeta^2-5.
\]
Furthermore, we have $H_4' \equiv 25\zeta^3-25\zeta^2+25\zeta \pmod{D_n}$
and $N_{\Q (\zeta)/\Q}(25\zeta^3-25\zeta^2+25\zeta)=5^8$.  
Therefore $H_4'$ is prime to $D_n$, and hence $H_4$ is also prime to $D_n$. Furthermore, since $t\equiv -(n^2+3n+4) \pmod{\delta_n}$,
we have $H_4 \equiv 0 \pmod{\delta_n}$ and hence  $H_4 =\delta_n H_4''$ with $H_4'' \in \Z [\zeta]$.
Since $\Delta_n =A_nB_nC_nD_n$ and $\delta_n$ are coprime and $H_4 \equiv u^3A_n B_n C_nH_4'\pmod{\Delta_n}$, it follows that
$A_nB_nC_n$ divides $H_4''$ in $\Z [\zeta]$, and $H_4''/(A_nB_nC_n)$ and $D_n$ are coprime.
Put $z:=H_4''/(A_n B_nC_n) \, (\in \Z[\zeta])$.  Then it follows that $z$ is prime to $D_n$ and  $\nu (g_4)=n^2
\delta_n A_n B_n C_n z$. %\qed
\end{itemize}
%\qed
\end{proof}

The next lemma can be obtained by using  Lemma~\ref{lem:alpha-A} for some particular divisors of $\Delta_n =ab^2c^3d^4e^5$.
%%%%%%%%%%%%%%%%%%
\begin{lem}\label{lem:alpha123}
Assume  $5\nmid n$.
\begin{enumerate}
\item[(1)]
Let $bc^2d^3e^3=p_1 \cdots p_k$ ($p_1,\cdots,p_k$ are not necessarily different prime numbers).
Then, $p_i \, (i\in  \{1,\ldots, k\})$ 
 decomposes into $p_i= \prod_{t=1}^4 \pi_i^{(t)}$ with $\pi_i^{(1)}|A_n $ in $\Z [\zeta]$ where 
 $ \pi_i^{(t)} $ are prime elements of $\Z [\zeta]$.
 Furthermore,  let $\lambda_1 := \prod_{i=1}^k \pi_i^{(1)} \, (\in \Z [\zeta ])$, $u_{\lambda_1}\, (\in \Z[\zeta]^{\times})$  the unit defined before Lemma~\ref{lem:prime-ideal}
 and $\alpha_1:=u_{\lambda_1} \lambda_1$.
 Then we have  $bc^2d^3e^3=N_{\Q (\zeta)/\Q} (\alpha_1)$,
 $\alpha_1 \equiv 1 \pmod{(1-\zeta)}$ and  $\alpha_1 |A_n$
 in $\Z[\zeta]$.
 \item[(2)]
Let $bcd^2e^2=p_1 \cdots p_k$ ($p_1,\cdots,p_k$ are not necessarily different prime numbers).
Then, $p_i \, (i\in  \{1,\ldots, k\})$ 
 decomposes into $p_i= \prod_{t=1}^4 \pi_i^{(t)}$ with $\pi_i^{(2)}|B_n $ in $\Z [\zeta]$ where 
 $ \pi_i^{(t)} $ are prime elements of $\Z [\zeta]$.
 Furthermore,  let $\lambda_2 := \prod_{i=1}^k \pi_i^{(2)} \, (\in \Z [\zeta ])$, $u_{\lambda_2}\, (\in \Z[\zeta]^{\times})$  the unit defined before Lemma~\ref{lem:prime-ideal}
 and $\alpha_2:=u_{\lambda_2} \lambda_2$.
 Then we have  $bcd^2e^2=N_{\Q (\zeta)/\Q} (\alpha_2)$,
 $\alpha_2 \equiv 1 \pmod{(1-\zeta)}$ and  $\alpha_2 |B_n$
 in $\Z[\zeta]$.
 \item[(3)]
Let $cde=p_1 \cdots p_k$ ($p_1,\cdots,p_k$ are not necessarily different prime numbers).
Then, $p_i \, (i\in  \{1,\ldots, k\})$ 
 decomposes into $p_i= \prod_{t=1}^4 \pi_i^{(t)}$ with $\pi_i^{(3)}|C_n $ in $\Z [\zeta]$ where 
 $ \pi_i^{(t)} $ are prime elements of $\Z [\zeta]$.
 Furthermore,  let $\lambda_3 := \prod_{i=1}^k \pi_i^{(3)} \, (\in \Z [\zeta ])$, $u_{\lambda_3}\, (\in \Z[\zeta]^{\times})$  the unit defined before Lemma~\ref{lem:prime-ideal}
 and $\alpha_3:=u_{\lambda_3} \lambda_3$.
 Then we have  $cde=N_{\Q (\zeta)/\Q} (\alpha_3)$,
 $\alpha_3 \equiv 1 \pmod{(1-\zeta)}$ and  $\alpha_3 |C_n$
 in $\Z[\zeta]$.
\end{enumerate}
\end{lem}

The following proposition gives the generator of the ideal $\nu (I)$ of (\ref{eq:nu-I}).
\begin{prop}\label{prop:gen-nu}
Assume  $5\nmid n$ and  let $\alpha_1, \alpha_2, \alpha_3 \, (\in \Z [\zeta] )$  be the elements given by Lemma~\ref{lem:alpha123}.
Then we have $\nu (I) =(n^2 \delta_n \alpha_1 \alpha_2 \alpha_3)$.
\end{prop}
\begin{proof}\
From (\ref{eq:nu-I}) and Lemma~\ref{lem:nu(g)},  we have
\begin{align}
\nu (I) & = (\nu (g_0), \nu (g_1), \nu (g_2) ,\nu(g_3), \nu (g_4)) \label{eq:fact-nu(I)}  \\
& = (n^2 \delta_n )(bc^2d^3e^3, A_n bcd^2e^2x, A_nB_n cdey, A_nB_nC_nz)  \notag
\end{align}
with $x,y$ and $z \, (\in \Z [\zeta])$ being prime to $B_n, C_n$ and $D_n$, respectively.
Since $\alpha_1 \alpha_2 \alpha_3$ divides all $
bc^2d^3e^3, A_n bcd^2e^2x, A_nB_n cdey, A_nB_nC_nz$, and the ideal generated by 
\[bc^2d^3e^3/(\alpha_1 \alpha_2 \alpha_3),\ A_n bcd^2e^2x/(\alpha_1 \alpha_2 \alpha_3),\ 
A_nB_n cdey/(\alpha_1 \alpha_2 \alpha_3),\ A_nB_nC_nz/(\alpha_1 \alpha_2 \alpha_3)
\]
is equal to $(1)={\mathcal O}_{K_n}$, 
we have 
\[
(bc^2d^3e^3, A_n bcd^2e^2x, A_nB_n cdey, A_nB_nC_nz)=(\alpha_1 \alpha_2 \alpha_3),
\]
and it follows that $\nu (I)=(n^2 \delta_n \alpha_1 \alpha_2 \alpha_3)$.
%\qed
\end{proof}
\begin{theo}\label{theo:main}
Let $n$ be an integer satisfying $5 \nmid n$ and  $\alpha_1, \alpha_2, \alpha_3 \, (\in \Z[\zeta])$ the elements
given in Lemma~\ref{lem:alpha123}.
 Put $\alpha_1 \alpha_2 \alpha_3 =\beta_0 +\beta_1 \zeta+\beta_2 \zeta^2+\beta_3 \zeta^3$ with
$\beta_0,\beta_1,\beta_2,\beta_3 \in \Z$ and 
\[
m:=\dfrac{1}{5} \left(\left( \frac{n}{5} \right) bc^2d^3e^4-n^2 (\beta_0+\beta_1+\beta_2+\beta_3) \right) \quad  (\in \Z).
\]
Then
\[
\dfrac{1}{bc^2d^3e^4} (\beta_0 \rho +\beta_1 \rho^{(1)}+\beta_2 \rho^{(2)} +\beta_3 \rho^{(3)} -m)
\]
is a generator of a normal integral basis of Emma Lehmer's  quintic field $K_n$.
\end{theo}
\begin{proof}\
Let $x:=\delta_n n^2 (\beta_0 +\beta_1 \sigma +\beta_2 \sigma^2+\beta_3 \sigma^3) \, (\in \Z[G])$ and  $\beta :=
\alpha_1 \alpha_2 \alpha_3=\beta_0+\beta_1 \zeta+\beta_2\zeta^2+\beta_3\zeta^3 \,
(\in \Z [\zeta])$. We have $\nu (x) =\delta_n n^2 \beta$. Furthermore,  from Proposition~\ref{prop:gen-nu}, 
for the generator $g$ of the ideal $I=(g)=(g_0,g_1,g_2,g_3,g_4)$, we have
$\nu (g)=v\delta_n n^2 \beta$ with $v\in \Z [\zeta]^{\times}$. Put 
\[
v_0 := \begin{cases}
1 & (v\equiv \pm 1 \pmod{ (1-\zeta)}, \\
\dfrac{1-\zeta^2}{1-\zeta}=1+\zeta & (v\equiv \pm 2 \pmod{(1-\zeta)}.
\end{cases}
\]
Since $v_0 \, (\in \Z [\zeta]^{\times})$ satisfies $v^{-1} v_0 \equiv \pm 1 \pmod{(1-\zeta)}$,
there exists $\xi \in \Z [G]^{\times}$ satisfying $\nu (\xi)=v^{-1} v_0 $ (\cite[p133, Theorem~1.6]{AF}).
Put 
\[
\tau := \begin{cases}
1 & (v\equiv \pm 1 \pmod{ (1-\zeta)},\\
1+\sigma &  (v\equiv \pm 2 \pmod{(1-\zeta)}.
\end{cases}
\]
Then we have
\[
\nu (\xi g)=\nu (\xi)\nu(g)=v_0 \delta_n n^2 \beta=\nu (\tau) \nu (x)=\nu (\tau x),
\]
and hence 
\[
\xi g-\tau x \ \in \text{Ker}(\nu)=(1+\sigma+\sigma^2+\sigma^3+\sigma^4),
\]
(\cite[Theorem~1.4]{AF}). 
It follows  that there exists $m' \in \Z$ satisfying
\begin{equation}\label{eq:rho/ell}
(\xi g- \tau x) \cdot \dfrac{\rho}{\ell} =m' \text{Tr} \left(\dfrac{\rho}{\ell} \right),
\end{equation}
where $\ell =-bc^2d^3e^4 \delta_n n^2$ and Tr is the trace map from $K_n$ to $\Q$.
Let $\alpha =g\cdot (\rho/\ell)$ be the generator of an NIB in  (\ref{eq:alpha-g}). From (\ref{eq:rho/ell}) we have 
\begin{align}
\xi \cdot \alpha & =(\xi g)\cdot \frac{\rho}{\ell} \notag \\
& =(\tau x)\cdot \frac{\rho}{\ell }+m'\mathrm{Tr}\left(\frac{\rho}{\ell}\right) 
\label{eq:xi-alpha}\\
& = \dfrac{n^2}{\ell} (\delta_n \tau \cdot  (\beta_0 \rho +\beta_1 \rho^{(1) }+\beta_2 \rho^{(2)} +\beta_3 \rho^{(3)})-m' ) . \notag
\end{align}
Since $\xi \cdot \alpha$ is a generator of an NIB, we have Tr$(\xi \cdot \alpha)=\pm 1$. 
Furthermore, we have Tr$(\rho^{(i)}) =-n^2\, (i \in \Z/5\Z)$. Taking the traces on both sides of (\ref{eq:xi-alpha}) and multiplying
them by $bc^2d^3e^4 \delta_n$,
it follows that
\begin{equation}\label{eq:trace-xi-alpha}
\pm bc^2d^3e^4 \delta_n =\delta_n sn^2 (\beta_0+\beta_1+\beta_2+\beta_3)+5m',\ s:=\begin{cases}
1 & (v\equiv \pm 1 \pmod{(1-\zeta)}), \\
2 & (v\equiv \pm 2 \pmod{(1-\zeta)}).
\end{cases}
\end{equation}
We can write $m'=\delta_n m$ with $m \in \mathcal O_{K_n}$ from (\ref{eq:xi-alpha}), and  $m=m'/\delta_n \in \Q \cap \mathcal O_{K_n }=\Z$.
It follows from (\ref{eq:trace-xi-alpha}) that 
\begin{equation}\label{eq:trace-xi-alpha-m}
\pm bc^2d^3e^4=sn^2(\beta_0 +\beta_1+\beta_2+\beta_3)+5m.
\end{equation}
Since $\alpha_1 \equiv \alpha_2 \equiv \alpha_3\equiv 1
\pmod{(1-\zeta)}$, we obtain
\[
\beta_0+\beta_1+\beta_2+\beta_3 \equiv \alpha_1\alpha_2\alpha_3 \equiv 1 \pmod{(1-\zeta)}.
\]
From $\beta_0 , \beta_1,\beta_2,\beta_3 \in \mathbb Z$, we have $\beta_0+\beta_1+\beta_2+\beta_3 \equiv 1\pmod{5}$,
and hence we obtain  $s\equiv \pm 1\pmod{5}$ from  $b\equiv c\equiv d\equiv e\equiv 1 \pmod{5}$ \cite[Lemma~2.1.1]{Je} and  (\ref{eq:trace-xi-alpha-m}). 
Therefore, we have $s=1$ ($v\equiv \pm 1 \pmod{(1-\zeta)}$ and $\tau=1$), and if we let $\varepsilon =\pm 1 $ be the
sign on the left-hand side of (\ref{eq:trace-xi-alpha-m}), then we have $\varepsilon \equiv n^2 \pmod{5}$, which implies $\varepsilon=\left(\frac{n}{5} \right)$.
We  conclude from (\ref{eq:trace-xi-alpha-m})
that
\[
m=\frac{1}{5} \left( \left( \frac{n}{5} \right)  bc^2d^3e^4 -n^2(\beta_0+\beta_1+\beta_2+\beta_3)\right) \ (\in \Z).
\]
From (\ref{eq:xi-alpha}) and $m'=\delta_n m$, it follows that $-(\beta_0 \rho +\beta_1 \rho^{(1)} +\beta_2 \rho^{(2)} +\beta_3 \rho^{(3)} -m)
/bc^2d^3e^4$ and its $-1$ times are both generators of an NIB.
The proof is complete.
%\qed
\end{proof}
\begin{rem}\label{rem:alpha123}
For the integers $\beta_0,\beta_1,\beta_2,\beta_3$ in Theorem~\ref{theo:main}, we have 
$N_{\Q (\zeta)/\Q} (\beta_0+\beta_1 \zeta +\beta_2 \zeta^2 +\beta_3 \zeta^3)=N_{\Q (\zeta)/\Q}
(\alpha_1\alpha_2\alpha_3)=b^2c^4d^6e^6$ from Lemma~\ref{lem:alpha123}.
\end{rem}
\begin{cor}\label{cor:NIB-squarefree}
Let $n$ be an integer satisfying $5 \nmid n$.   Then $K_n$ has a generator of a normal integral basis of the form $v+w \rho \ (v,w\in \Z)$
if and only if  $\Delta_n$ is square-free. Furthermore, in this case, the integers $v$ and $w$ are given by 
 $w=\pm1,\, v=w\left(n^2-\left(\frac{n}{5} \right)\right)/5$.
\end{cor}
\begin{proof}\
First, we assume that  $\Delta_n$ is square-free, then we have $b=c=d=e=1$, and hence $\alpha_1=1,\, \alpha_2=1$ and $\alpha_3=1$ 
from Lemma~\ref{lem:alpha123}. Therefore $\beta_0, \beta_1, \beta_2$ and $\beta_3$ in Theorem~\ref{theo:main} are given by $\beta_0=1,\, \beta_1=\beta_2=\beta_3=0$.
Therefore, the integer $m$ in Theorem~\ref{theo:main} is given by $m=\left( \left( \frac{n}{5} \right)-n^2 \right)/5$. It follows from Theorem~\ref{theo:main} that $\alpha:=\left(  n^2 -\left(\frac{n}{5} \right) \right)/5 +\rho$ is a generator of an NIB of $K_n$. 

Next, we assume that $K_n$ has a generator of an NIB of the form $v+w\rho \ (v,w \in \Z)$. Since
\[
d (v+w\rho, v+w \rho^{(1)}, v+w\rho^{(2)} , v+w \rho^{(3)}, v+w \rho^{(4)}) =w^8 (5v-wn^2)^2 \Delta_n^4
\]
and (\ref{eq:cond}), (\ref{eq:disc}),  it follows that $w^8=(5v-wn^2)^2=1$ and $\Delta_n$ is square-free. 
%\qed
\end{proof}
\begin{ex}\label{ex:1to1000}

Table~\ref{table:1000} shows a generator of an NIB of $K_n$ such that $n$ satisfies $1\leq n\leq 1000,\, 5\nmid n$ and
$\Delta_n$ is not square-free.
For example, for $n=14$, $\Delta_n=11 \cdot 71^2$ and $\alpha_1,\alpha_2,\alpha_3 \, (\in \mathbb Z[\zeta])$ in
 Theorem~\ref{theo:main}
are given by $\alpha_1=3\zeta^2 + \zeta + 2 , \, \alpha_2=\zeta^3 + 3\zeta + 2 ,\, \alpha_3=  1 $. Since 
$\alpha_1 \alpha_2 \alpha_3= 10 \zeta^3 + 8 \zeta^2 + 7\zeta + 6$, we have $\beta_0=6,\, \beta_1=7,\, \beta_2=8,\, 
\beta_3=10$ and $m=\left(\left( \frac{n}{5} \right) bc^2d^3e^4-n^2 (\beta_0+\beta_1+\beta_2+\beta_3) \right)/5=-1201$.
Therefore, $(6\rho+7\rho^{(1)} +8\rho^{(2)}
+10\rho^{(3)} +1201)/71$
is an NIB of $K_n $ for $n=14$  by Theorem~\ref{theo:main}.
\end{ex}
\begin{ex}\label{ex:d-not-1}
The smallest positive integer with $d \ne 1$ is $n=2888$.
In this case, we have $\Delta_n=11^4 \cdot 4759595441$ and $\mathfrak f_{K_n}=11 \cdot 4759595441$,
and $ (-16\rho-6\rho^{(1)} -26\rho^{(2)}
-41\rho^{(3)} -148461417)/11^3$
is an NIB of $K_n $ .
\end{ex}
\begin{ex}\label{ex:e-not-1}
The smallest positive integer with $e \ne 1$ is $n=7721$.
In this case, we have $\Delta_n=11^5\cdot 26501 \cdot 833201$ and $\mathfrak f_{K_n}=26501\cdot 833201$,
and $ (-10\rho+6\rho^{(1)} -35\rho^{(2)}
-20\rho^{(3)} -703446252)/11^4$
is an NIB of $K_n $ .
\end{ex}
\begin{ex}\label{two-not-1}
The smallest positive integer for which two of $b,c,d$ and $e$ are not $1$ is  $n=40846$.
In this case, we have $\Delta_n=11^4 \cdot 31^2\cdot 197859618251$ and $\mathfrak f_{K_n}=11 \cdot 31 \cdot 197859618251$,
and $(-211\rho-96\rho^{(1)} -158\rho^{(2)}
-14\rho^{(3)} -159832317845)/(11^3 \cdot 31)$
is an NIB of $K_n $ .
\end{ex}
%%%%%%%%%%%%%%%%%%%%%%%%%%%%
\begin{table}[H]
\caption{$1\leq n\leq 1000,\, 5\nmid n$ and $\Delta_n$ is not square-free }
\label{table:1000}
 \begin{center}
{\scriptsize
 \begin{tabular}{|c||c|c|c|}
 \hline
\rule{0pt}{4mm}   $n$ & $\Delta_n$  & $\mathfrak f_{K_n}$ & A generator of NIB \\  \hline  \hline 
\rule{0pt}{4mm}   $14$ & $11\cdot71^2$ & $11\cdot 71$ &  $\frac{1}{71} (6\rho+7\rho^{(1)} +8\rho^{(2)}
+10\rho^{(3)} +1201)$ \\  \hline
\rule{0pt}{4mm}   $44$  & $61 \cdot 41^3$ &   $41\cdot 61$& 
$\frac{1}{ 41^2} ( 28\rho +39 \rho^{(1)} +36\rho^{(2)} +48\rho^{(3)}+58131 ) $ \\ \hline
\rule{0pt}{4mm}   $69$ & $201511\cdot 11^2$  & $11\cdot 201511 $ & 
$\frac{1}{11} (\rho-3\rho^{(1)} -\rho^{(2)}
-\rho^{(3)} -3811)$ \\ \hline
\rule{0pt}{4mm}   $71$ &  $7331\cdot 61^2$ & $61 \cdot 7331$ & 
$\frac{1}{61} (-3\rho -2\rho^{(2)}
+6\rho^{(3)} +996)$  \\ \hline
\rule{0pt}{4mm}   $83$ & $11\cdot 2141^2$  & $11 \cdot 2141$ & 
$\frac{1}{2141} (16\rho +2 \rho^{(1)}-37\rho^{(2)}
+10\rho^{(3)} -11972)$ \\ \hline
 \rule{0pt}{4mm}   $86$  &  $31 \cdot 15461\cdot 11^2$ &   $11 \cdot 31 \cdot 15461$& 
$\frac{1}{11} (4\rho +3\rho^{(1)}+2\rho^{(2)}
+2\rho^{(3)} +16269)$ \\ \hline
 \rule{0pt}{4mm}   $98$ & $191\cdot4201\cdot 11^2$ &  $11\cdot191\cdot4201$ & 
$\frac{1}{11}  (2\rho -2\rho^{(2)}
+\rho^{(3)} +1923)$
\\ \hline
 \rule{0pt}{4mm}   $207$  & $15545731\cdot 11^2$ & $11\cdot 15545731$  & 
$\frac{1}{11} ( 4\rho +3 \rho^{(1)} +2\rho^{(2)} +2\rho^{(3)} +94270) $ \\ \hline
 \rule{0pt}{4mm}   $219$ & $19450411\cdot 11^2$ &  $11\cdot 19450411$ &
$\frac{1}{11} ( 2\rho  -2 \rho^{(2)} +\rho^{(3)} +9590) $  \\ \hline
\rule{0pt}{4mm}    $226$  & $101\cdot19841\cdot 11^3$  &  $11\cdot 101 \cdot 19841$ & 
$\frac{1}{11^2} ( -2\rho -9 \rho^{(1)} -6\rho^{(2)} -12\rho^{(3)} -296265) $ \\ \hline
 \rule{0pt}{4mm}   $276$  & $61\cdot100801\cdot 31^2$ & $31 \cdot 61 \cdot 100801$ & 
$\frac{1}{31} ( \rho  -2\rho^{(1)} -2\rho^{(2)} +4\rho^{(3)} +15229) $   \\ \hline
 \rule{0pt}{4mm}   $311$ & $131\cdot 461\cdot1301\cdot 11^2$  & $11\cdot 131 \cdot 461 \cdot 1301$ &  
$\frac{1}{11} ( \rho-3  \rho^{(1)} -\rho^{(2)} -\rho^{(3)} -77379) $  \\ \hline
 \rule{0pt}{4mm}   $328$ &  $97127081\cdot 11^2$ & $11 \cdot 97127081$ &  
$\frac{1}{11} ( 4\rho  +3\rho^{(1)} +2\rho^{(2)} +2\rho^{(3)} +236687) $  \\ \hline
 \rule{0pt}{4mm}   $347$  & $121562411\cdot 11^2$ &  $11\cdot 121562411$ & 
$\frac{1}{11} ( 2\rho  +\rho^{(2)} -2\rho^{(3)}+24084 ) $   \\ \hline
\rule{0pt}{4mm}    $432$ &  $291193681 \cdot 11^2$ & $11\cdot 291193681$ & 
$\frac{1}{11} ( \rho  -3\rho^{(1)} -\rho^{(2)} -\rho^{(3)}-149297 ) $   \\ \hline
\rule{0pt}{4mm}    $449$  &  $30877981 \cdot 11^3$ & $11\cdot 30877981$ &  
$\frac{1}{ 11^2} ( 10\rho  +6\rho^{(1)} +12\rho^{(2)} +3\rho^{(3)}+1249902 ) $ \\ \hline
\rule{0pt}{4mm}    $461$  & $377340791\cdot 11^2$   & $11 \cdot 377340791$ &  
$\frac{1}{11} ( 2\rho  -2\rho^{(2)} +\rho^{(3)} +42502) $ \\ \hline
\rule{0pt}{4mm}    $468$ &  $400721701 \cdot 11^2$  &  $11 \cdot 400721701$ &  
$\frac{1}{11} ( 2\rho  +\rho^{(2)} -2\rho^{(3)} +43807) $ \\ \hline
\rule{0pt}{4mm}     $484$ & $131\cdot440431\cdot 31^2$  &  $31\cdot 131 \cdot 440431$&  
$\frac{1}{31} ( -\rho  +3\rho^{(1)} -3\rho^{(2)} -3\rho^{(3)} -187411) $ \\ \hline
\rule{0pt}{4mm}    $544$ & $41\cdot 2243281 \cdot 31^2$  &  $31\cdot 41 \cdot 2243281$ &  
$\frac{1}{31} ( 6\rho +3 \rho^{(1)} +2\rho^{(3)} +651053) $ \\ \hline
\rule{0pt}{4mm}     $553$ &  $779911631 \cdot 11^2$ &  $11 \cdot 779911631$ & 
$\frac{1}{11} ( \rho  -3\rho^{(1)} -\rho^{(2)} -\rho^{(3)} -244645) $ \\ \hline
\rule{0pt}{4mm}   $582$  & $31\cdot71\cdot571\cdot761 \cdot 11^2$  &  $11 \cdot 31 \cdot 71 \cdot 571 \cdot 761$ &  
$\frac{1}{11} ( 2\rho  -2\rho^{(2)} +\rho^{(3)}+67747 ) $ \\ \hline
\rule{0pt}{4mm}     $589$ &  $7151\cdot140281 \cdot 11^2$ & $11\cdot 7151 \cdot 140281$ &  
$\frac{1}{11} ( 2\rho +\rho^{(2)} -2\rho^{(3)} +69382) $ \\ \hline
\rule{0pt}{4mm}     $613$ &  $1091\cdot 135781\cdot 31^2$ & $31 \cdot 1091 \cdot 135781$ &  
$\frac{1}{31} ( -6\rho  -4\rho^{(1)} -6\rho^{(2)} -3\rho^{(3)} -1427916) $ \\ \hline
\rule{0pt}{4mm}     $674$ & $20411 \cdot 84181 \cdot 11^2$   & $11 \cdot 20411 \cdot 84181$ &  
$\frac{1}{11} ( \rho  -3\rho^{(1)} -\rho^{(2)} -\rho^{(3)} -363423) $ \\ \hline
\rule{0pt}{4mm}     $691$ &  $1897892411\cdot 11^2$ &  $11 \cdot 1897892411$& 
$\frac{1}{11} ( 4\rho  +3\rho^{(1)}+2 \rho^{(2)} +2\rho^{(3)} +1050456) $ \\ \hline
\rule{0pt}{4mm}     $703$ &  $61\cdot 101\cdot 311 \cdot 1061 \cdot 11^2$  &  $11\cdot 61 \cdot 101 \cdot 311 \cdot 1061$&  
$\frac{1}{11} ( 2\rho  -2\rho^{(2)} +\rho^{(3)} +98844) $  \\ \hline
\rule{0pt}{4mm}     $726$ &  $166407091\cdot 41^2$  &  $41 \cdot 166407091$ & 
$\frac{1}{41} ( -2\rho  -6\rho^{(1)} -4\rho^{(2)}-7\rho^{(3)} -2002897) $ \\ \hline
\rule{0pt}{4mm}     $812$ &  $48491 \cdot 74551 \cdot 11^2$ & $11 \cdot 48491 \cdot 74551$ & 
$\frac{1}{11} ( 4\rho  +3\rho^{(1)} +2\rho^{(2)} +2\rho^{(3)} +1450559) $ \\ \hline
\rule{0pt}{4mm}     $824$ &  $61\cdot 3271 \cdot 19211 \cdot 11^2$ & $11 \cdot 61 \cdot 3271 \cdot 19211$ & 
$\frac{1}{11} ( 2\rho  -2\rho^{(2)} +\rho^{(3)} +135793) $ \\ \hline
\rule{0pt}{4mm}     $831$ &  $41 \cdot 571 \cdot 169361 \cdot 11^2$ & $11\cdot 41 \cdot 571 \cdot 169361$ & 
$\frac{1}{11} ( 2\rho  +\rho^{(2)} -2\rho^{(3)} +138110) $ \\ \hline
\rule{0pt}{4mm}     $916$ & $ 31\cdot 17155921\cdot 11^3$  & $11\cdot 31\cdot 17155921$ &
$\frac{1}{ 11^2} (4 \rho -6 \rho^{(1)} -3\rho^{(2)}+6 \rho^{(3)} +167787) $ \\ \hline
\rule{0pt}{4mm}     $933$ & $101\cdot 62337371\cdot  11^2$  & $11 \cdot 101 \cdot 62337371$ & 
$\frac{1}{11} ( 4\rho +3 \rho^{(1)} +2\rho^{(2)} +2\rho^{(3)}+1915078 ) $ \\ \hline
\rule{0pt}{4mm}     $952$ & $101\cdot 811\cdot 83311 \cdot 11^2$  &  $11 \cdot 101 \cdot 811 \cdot 83311$& 
$\frac{1}{11} ( 2\rho  + \rho^{(2)} -2\rho^{(3)} +181263) $  \\ \hline
\end{tabular}
}
\end{center}
\end{table}
%\end{ex}
%
%\begin{ex}\label{ex:d-not-1}
%The smallest positive integer with $d \ne 1$ is $n=2888$.
%In this case, we have $\Delta_n=11^4 \cdot 4759595441$ and $\mathfrak f_{K_n}=11 \cdot 4759595441$,
%and $ (-16\rho-6\rho^{(1)} -26\rho^{(2)}
%-41\rho^{(3)} -148461417)/11^3$
%is an NIB of $K_n $ .
%\end{ex}
%
%
%\begin{ex}\label{ex:e-not-1}
%The smallest positive integer with $e \ne 1$ is $n=7721$.
%In this case, we have $\Delta_n=11^5\cdot 26501 \cdot 833201$ and $\mathfrak f_{K_n}=26501\cdot 833201$,
%and $ (-10\rho+6\rho^{(1)} -35\rho^{(2)}
%-20\rho^{(3)} -703446252)/11^4$
%is an NIB of $K_n $ .
%\end{ex}
%
%\begin{ex}\label{two-not-1}
%The smallest positive integer for which two of $b,c,d$ and $e$ are not $1$ is  $n=40846$.
%In this case, we have $\Delta_n=11^4 \cdot 31^2\cdot 197859618251$ and $\mathfrak f_{K_n}=11 \cdot 31 \cdot 197859618251$,
%and $(-211\rho-96\rho^{(1)} -158\rho^{(2)}
%-14\rho^{(3)} -159832317845)/(11^3 \cdot 31)$
%is an NIB of $K_n $ .
%\end{ex}
\section{All normal integral bases}

Let $\alpha$ be a generator of an NIB of $K_n$. Then the set of all generators of NIBs is $\{ \pm \sigma^{\ell} (1-\sigma^2-\sigma^3 )^k.\alpha 
\, |\, \ell \in \Z/5\Z, k\in \Z \}$ by (\ref{eq:ZG}).
Davis, Eloff, Spearman and Williams \cite{DES,ESW2} determined all NIBs of $K_n$ for $n=-1$ and parametrized them using the Fibonacci and Lucas numbers.
In this section, we follow their method and  find  all NIBs of $L_n$ for the general $n$. Let $\lambda := (3+\sqrt{5})/2 =( (1+\sqrt{5})/2)^2$ be the square of the fundamental unit
of $\Q (\sqrt{5})$, and $\overline{\lambda}:=(3-\sqrt{5})/2$ its conjugate.
For any integer $ k$, define the sequences $a_k, b_k$, and $c_k$ by
\begin{align}
a_k & := \frac{1}{5} ( (-1)^k +2 (\lambda^k +\overline{\lambda}^k ) ), \notag \\
b_k & := \frac{1}{2} (a_k-a_{k-1}), \label{eq:abc} \\
c_k & :=\frac{1}{2} (a_{k+1} -a_k) \quad  (=b_{k+1}). \notag 
\end{align}
For example,  $a_0=a_1=1,\, a_2=3,\, b_0=b_1=0,\, b_2=1,\, c_0=0,\, c_1=1,\, c_2=2 $.
Let $L_k$ be the Lucas number defined by $L_0:=2,\, L_1=1$ and $L_k=L_{k-1}+L_{k-2}\ (k\in \mathbb Z)$. We can show the following lemma   from the formula:
$L_k= ( (1+\sqrt{5})/2)^k +((1-\sqrt{5})/2)^k$.
\begin{lem}\label{lem:abc-L}
For any $k\in \Z$, we have $a_k =((-1)^k +2L_{2k})/5,\, b_k =((-1)^k +L_{2k-1})/5$ and $ c_k=((-1)^{k+1} +L_{2k+1} )/5$.
\end{lem}
Using the above lemma, we can also prove the following lemma.
\begin{lem}\label{lem:abc-rel}
For any $k\in \Z$, we have $a_k,b_k,c_k \in \Z$ and $a_k=a_{-k},\, b_{-k}=-c_k,\, a_{k+1}-2a_k -2a_{k-1}+a_{k-2}=0,\, b_{k+1} -2b_k -2b_{k-1}+b_{k-2}=0,\,
c_{k+1}-2c_k -2c_{k-1}+c_{k-2}=0$.
\end{lem}
Let $\rho =\rho_n$ be a root of 
the quintic polynomial $f_n(X)$ and $G=\mathrm{Gal}(K_n/\Q)=\langle \sigma \rangle$.
The action of $(1-\sigma^2-\sigma^3)^k $ on $\rho$  is given by  the sequences $a_k, b_k$ and $c_k$ as follows.
\begin{lem}\label{lem:rho-action}
For any $k\in \Z$, we have the following.
\begin{align*}
(1-\sigma^2-\sigma^3)^k  \cdot \rho  & =a_k \rho +b_k (\rho^{(1)}+\rho^{(4)} )-c_k (\rho^{(2)}+\rho^{(3)} ) \\
& =(a_k +b_k (\sigma+\sigma^4) -c_k (\sigma^2+\sigma^3)) \cdot \rho
\end{align*}
\end{lem}
\begin{proof}\
From (\ref{eq:abc}) and Lemma~\ref{lem:abc-rel}, we have $a_k+2c_k=a_{k+1},\ c_k=b_{k+1},\ a_k+b_k+c_k=c_{k+1}$ for any $k\in \mathbb Z$.
The assertion can be shown by induction on $k$.
%\qed
\end{proof}
By acting  $\pm \sigma^{\ell} (1-\sigma^2-\sigma^3)^k \, (\ell \in \Z/5\Z, k\in \Z)$  to the generator of an NIB obtained in Theorem~\ref{theo:main},
we can obtain all  generators of NIBs.
\begin{theo}\label{theo:all-NIB}
Let $n$ be an integer with $5\nmid n$ and $\beta_0, \beta_1, \beta_2, \beta_3, m \, (\in \Z)$ as in
Theorem~\ref{theo:main}. Put
\begin{align*}
\theta_0 (k) & :=a_k \beta_0 +b_k \beta_1-c_k\beta_2-c_k\beta_3, \\
\theta_1(k) & :=b_k \beta_0 +a_k \beta_1+b_k \beta_2 -c_k \beta_3, \\
\theta_2(k) & :=-c_k \beta_0 +b_k \beta_1+a_k \beta_2+b_k \beta_3, \\
\theta_3 (k) & :=-c_k \beta_0-c_k \beta_1+b_k \beta_2+a_k \beta_3, \\
\theta_4(k) & :=b_k\beta_0-c_k\beta_1-c_k \beta_2+b_k \beta_3.
\end{align*}
Then we have 
\[
 \{ x \ |\ \text{a generator of an NIB of $K_n$} \}  =\{ \pm \sigma^{\ell} \cdot  \xi_k \ |\ \ell \in \mathbb Z/5\mathbb Z,\ k\in \mathbb Z \},
\]
where
\[
\xi_k:= \frac{1}{bc^2d^3e^4} \left( \sum_{t=0}^4 \theta_t(k) \rho^{(t)}-(-1)^km \right)
\]
\end{theo}
\begin{proof}\
Let 
\begin{align*}
\alpha & :=\dfrac{1}{bc^2d^3e^4} ( \beta_0 \rho +\beta_1 \rho^{(1)}+\beta_2 \rho^{(2)} +\beta_3 \rho^{(3)} -m) \\
& = \dfrac{1}{bc^2d^3e^4} ((\beta_0  +\beta_1 \sigma +\beta_2 \sigma^2 +\beta_3 \sigma^3 ) \cdot \rho -m) 
\end{align*}
be the generator of an NIB obtained by Theorem~\ref{theo:main}. 
It follows from Lemma~\ref{lem:rho-action} that
\begin{align*}
& \pm \sigma^{\ell} (1-\sigma^2-\sigma^3)^k \cdot  \alpha \\
& =\pm \frac{1}{bc^2d^3e^4} \sigma^{\ell} 
( (\beta_0+\beta_1 \sigma +\beta_2 \sigma^2 +\beta_3 \sigma^3 )(a_k+b_k (\sigma+\sigma^4)-c_k(\sigma^2+\sigma^3))  \cdot \rho
-(-1)^km).
\end{align*}
Furthermore,  direct calculations  yield 
\begin{align*}
& (\beta_0+\beta_1 \sigma +\beta_2 \sigma^2 +\beta_3 \sigma^3 )(a_k+b_k (\sigma+\sigma^4)-c_k(\sigma^2+\sigma^3)) \\
& = \theta_0(k)+\theta_1(k)\sigma +\theta_2(k) \sigma^2 +\theta_3(k) \sigma^3+\theta_4(k)\sigma^4.
\end{align*}
The proof is complete.
%\qed
\end{proof}
\begin{cor}\label{cor:all-NIB-squarefree}
Let $n$ be an integer with $5\nmid n$ and $v:=\left(
n^2-\left(\frac{n}{5} \right)\right)/5$. Assume that $\Delta_n$ is square-free.
Then we have
\begin{align*}
& \{ x \, |\, \text{a generator of an NIB of $K_n$ } \}  =\{ \pm \sigma^{\ell}  \cdot  \xi_k \ |\ \ell \in \mathbb Z/5\mathbb Z, k\in \mathbb Z \},
\end{align*}
where $
\xi_k :=(-1)^k v+a_k \rho +b_k (\rho^{(1)}+\rho^{(4)} )-c_k (\rho^{(2)}+\rho^{(3)} )$.
\end{cor}
\begin{proof}\
It follows that $\beta_0=1,\, \beta_1=\beta_2=\beta_3=0$  from $ b=c=d=e=1$ since $\Delta_n$ is square-free.
Therefore, we obtain $\theta_0(k)=a_k,\, \theta_1(k)=\theta_4(k)=b_k,\, \theta_2(k)=\theta_3(k)=-c_k$.
From Theorem~\ref{theo:all-NIB},  the generators of all NIBs of $K_n$ are given by $\pm \sigma^{\ell} \cdot \alpha_k$ for some $\ell \in \Z/5\Z,\, k\in \mathbb Z$.
%\qed
\end{proof}
The following example is the case of  $n=-1$ which Davis, Eloff, Spearman and Williams considered in \cite{DES,ESW2}.
In this case, we have $K_n=\Q (\zeta_{11}+\zeta_{11}^{-1})$ and $f_n(X)=X^5+X^4-4X^3-3X^2+3X+1$ is the minimal polynomial of $\zeta_{11}+\zeta_{11}^{-1}$.
\begin{ex}\label{ex:n=-1}
Let $n=-1$, we have $f_n(X)=X^5+X^4-4X^3-3X^2+3X+1,\,  \Delta_n=11,\, a=11,\, b=c=d=e=1,\, \delta_n=1,\,\mathfrak f_{K_n}=11$ and $D_{K_n}=11^4$.
Furthermore, we have $\rho^{(1)}=\rho^4-4\rho^2+2,\ \rho^{(2)}=-\rho^4-\rho^3+3\rho^2+2\rho-1,\ \rho^{(3)} =\rho^2-2$ and
$\rho^{(4)} =\rho^3-3\rho$.
Since $v=\left(n^2-\left(\frac{n}{5} \right)\right)/5=0$, it follows from Corollary~\ref{cor:all-NIB-squarefree} that
\[
\{ x \, | \, \text{  a generator of an NIB of $K_n$} \} = \{ \pm \sigma^{\ell} \cdot \xi_k \, |\, 
\ell \in \Z/5\Z,\, k\in \Z \},
\]
where 
\begin{align*}
\xi_k &  = a_k \rho +b_k (\rho^{(1)}+\rho^{(4)} )-c_k (\rho^{(2)}-\rho^{(3)} )  \\
& = \frac{(-1)^k+2L_{2k}}{5} \rho +\frac{(-1)^k+L_{2k-1}}{5}
 (\rho^{(1)}+\rho^{(4)} )-\frac{(-1)^{k+1}+L_{2k+1}}{5} (\rho^{(2)}+\rho^{(3)} )\\
 & =\frac{1}{5} ( (-1)^{k+1}+2L_{2k-1}+3L_{2k+1} ) -L_{2k-1} \, \rho-\frac{4}{5} (L_{2k-1}+L_{2k+1})
 \rho^2 \\
 & \quad +\frac{1}{5} (L_{2k-1}+L_{2k+1}) \, \rho^3+ \frac{1}{5} (L_{2k-1}+L_{2k+1}) \, \rho^4.
 \end{align*}
\end{ex}
\begin{rem}\label{rem:n=-1} Let $F_k$ be the Fibonacci number defined by $F_0:=0,\,  F_1:=1$ and $ F_k=F_{k-1}+F_{k-2}$.
In \cite[theorem~1]{DES}, they showed that
\[
\{ x \, | \, \text{ a generator of an NIB of $\mathbb Q (\zeta_{11}+\zeta_{11}^{-1})$} \} = \{ \pm \sigma^{\ell} \cdot \gamma_k \, |\, 
\ell \in \Z/5\Z,\, k\in \Z \}
\]
where 
\[
\gamma_k :=\frac{1}{10}(25F_{2k}+(-1)^k L_{2k}-2) +\frac{1}{2} (-5F_{2k} +(-1)^k L_{2k})\, \rho -4F_{2k}\,  \rho^2+F_{2k}\, \rho^3 +F_{2k}\, \rho^4.
\]
Using formulae $F_{-k}=(-1)^{k+1} F_k,\, L_{-k}=(-1)^k L_k$ and
$L_{k-1}+L_{k+1}=5F_k$, it  can be shown that the relation between $\gamma_k$ and  $\xi_k$ in Example~\ref{ex:n=-1} 
 given by $\xi_k =(-1)^k \gamma_{(-1)^kk} $ holds for any $k\in \Z$.
\end{rem}
\begin{ex}\label{ex:NIBn=14}
Let $n=14$. We have $f_n(X)=X^5+196X^4-6814X^3+54507X^2+3678X+1,\,  \Delta_n=11\cdot 71^2,\, a=11,\, b=71,\,
c=d=e=1,\, \delta_n=7^2\cdot 79,\,\mathfrak f_{K_n}=11\cdot 71$ and $D_{K_n}=11^4\cdot 71^4$.
It follows from Theorem~\ref{theo:all-NIB} that
\[
\{ x \, | \, \text{  a generator of an NIB of $K_n$} \} = \{ \pm \sigma^{\ell} \cdot \xi_k \, |\, 
\ell \in \Z/5\Z,\, k\in \Z \},
\]
where 
\begin{align*}
\xi_k  = & (\theta_0(k)\rho+\theta_1(k)\rho^{(1)}+\theta_2(k)\rho^{(2)}+\theta_3(k)\rho^{(3)}+\theta_4(k)\rho^{(4)}+(-1)^k1201)/71,\\
\theta_0(k) = &\frac{31}{5}(-1)^k -L_{2k-1}-\frac{6}{5}L_{2k+1},  \\
\theta_1(k) =& \frac{31}{5} (-1)^k+ \frac{4}{5} L_{2k+1}, \\
\theta_2(k) = & \frac{31}{5}(-1)^k+ \frac{1}{5} L_{2k-1} +2 L_{2k+1},\\
\theta_3(k) = & \frac{31}{5} (-1)^k -\frac{12}{5} L_{2k-1} +\frac{7}{5} L_{2k+1} , \\
\theta_4(k)  =& \frac{31}{5} (-1)^k+ \frac{16}{5} L_{2k-1} -3 L_{2k+1} .
 \end{align*}
The following table shows $\xi_k$ for $k$ satisfying $-5\leq k\leq 5$.
\begin{table}[H]\label{table:n=14}
\caption{$\xi_k$ for $-5 \leq k\leq 5$}
 \begin{center}
{\scriptsize
 \begin{tabular}{|c||c|}
 \hline
\rule{0pt}{4mm}   $k$ & $\xi_k$  \\  \hline  \hline 
\rule{0pt}{4mm}   $-5$ &   $(71+1451\rho-304\rho^{(1)} -959\rho^{(2)} +1856\rho^{(3)}-2044\rho^{(4)} )/355$\\  \hline  \hline 
\rule{0pt}{4mm}   $-4$ &   $(-71+554\rho-116\rho^{(1)} -366\rho^{(2)} +709\rho^{(3)}-11\rho^{(4)} )/355$\\  \hline  \hline 
\rule{0pt}{4mm}   $-3$ &  $(71+211\rho-44\rho^{(1)} -139\rho^{(2)} +271\rho^{(3)}-299\rho^{(4)} )/355$\\  \hline  \hline 
\rule{0pt}{4mm}   $-2$ &   $(-71+79\rho-16\rho^{(1)} -51\rho^{(2)} +104\rho^{(3)}-116\rho^{(4)} )/355$ \\  \hline  \hline 
\rule{0pt}{4mm}   $-1$ &   $(71+26\rho-4 \rho^{(1)} -14\rho^{(2)} +41\rho^{(3)}-49\rho^{(4)} )/355$\\  \hline  \hline 
\rule{0pt}{4mm}   $0$ &  $(-71-\rho+4 \rho^{(1)} +9\rho^{(2)} +19\rho^{(3)}-31\rho^{(4)} )/355$\\  \hline  \hline 
\rule{0pt}{4mm}   $1$ &   $(71-29\rho+16\rho^{(1)} +41\rho^{(2)} +16\rho^{(3)}-44\rho^{(4)} )/355$ \\  \hline  \hline 
\rule{0pt}{4mm}   $2$ &  $(-71-86\rho+44\rho^{(1)} +114\rho^{(2)} +29\rho^{(3)}-101\rho^{(4)} )/355$ \\  \hline  \hline 
\rule{0pt}{4mm}   $3$ & $(71-229\rho+116\rho^{(1)} +301\rho^{(2)} +71\rho^{(3)}-259\rho^{(4)} )/355$ \\  \hline  \hline 
\rule{0pt}{4mm}   $4$ &  $(-71-601\rho+304\rho^{(1)} +789\rho^{(2)} +184\rho^{(3)}-676\rho^{(4)} )/355$ \\  \hline  \hline 
\rule{0pt}{4mm}   $5$ & $(71-1574\rho+796\rho^{(1)} +2066\rho^{(2)} +481\rho^{(3)}-1769\rho^{(4)} )/355$ \\  \hline  \hline 
\end{tabular}
}
\end{center}
\end{table}
\end{ex}

\end{document}